\documentclass[12pt]{amsart}
\usepackage[latin9]{inputenc}
\usepackage{xcolor}
\usepackage{amstext}
\usepackage{amsthm}
\usepackage{amssymb}
\usepackage[unicode=true,
 bookmarks=false,
 breaklinks=false,pdfborder={0 0 1},backref=section,colorlinks=false]{hyperref}
\usepackage{xcolor}
\usepackage{verbatim}
\usepackage{amstext}
\usepackage{amsthm}
\usepackage{eurosym}
\usepackage{latexsym}
\usepackage{amsthm}
\usepackage{amsfonts}
\usepackage{ulem}

\makeatletter
\numberwithin{equation}{section}
\numberwithin{figure}{section}
\numberwithin{equation}{section}
\numberwithin{figure}{section}
\numberwithin{equation}{section}
\newtheorem{theorem}{Theorem}

\newtheorem{lemma}[theorem]{Lemma}
\newtheorem{corollary}[theorem]{Corollary}
\newtheorem{proposition}[theorem]{Proposition}

\newtheorem{problem}[theorem]{Problem}
\theoremstyle{definition}
\newtheorem{definition}[theorem]{Definition}

\newtheorem*{acknowledgements*}{Acknowledgements}
\theoremstyle{remark}
\newtheorem{remark}[theorem]{Remark}

\numberwithin{theorem}{section}   
\thanks{The work of the first author was completed as a part of the implementation of the development program of the Scientific and Educational Mathematical Center Volga Federal District, agreement no. 075-02-2021-1393.}
\thanks{}
\thanks{}
\makeatother

\begin{document}
\title[A description of interpolation spaces]{A description of interpolation
spaces for quasi-Banach couples by real $K$-method}
\author{Sergey V. Astashkin}
\address{Astashkin:Department of Mathematics, Samara National Research
University, Moskovskoye shosse 34, 443086, Samara, Russia}
\email{astash56@mail.ru}
\author{Per G. Nilsson}
\address{Nilsson: Roslagsgatan 6, 113 55 Stockholm, Sweden}
\email{pgn@plntx.com}
\date{\today }

\begin{abstract}
The main aim of this paper is to develop a general approach, which allows to
extend the basics of Brudnyi-Kruglyak interpolation theory to the realm of
quasi-Banach lattices. We prove that all $K$-monotone quasi-Banach lattices
with respect to a $L$-convex quasi-Banach lattice couple have in fact a
stronger property of the so-called $K(p,q)$-monotonicity for some $0<q\leq
p\leq 1$, which allows us to get their description by the real $K$-method.
Moreover, we obtain a refined version of the $K$-divisibility property for
Banach lattice couples and then prove an appropriate version of this
property for $L$-convex quasi-Banach lattice couples. The results obtained
are applied to refine interpolation properties of couples of sequence $l^{p}$%
- and function $L^{p}$-spaces, considered for the full range $0<p<\infty $.
\end{abstract}

\maketitle

\section{Introduction}

\label{Intro}

It is well known that the interpolation theory of operators has numerous
valuable applications to various fields of mathematics: harmonic analysis,
Banach space theory, the theory of orthogonal series, noncommutative
analysis and many others. This feature is fully reflected in excellent books
devoted to this theory; see, for example, Bergh and Löfström \cite{BL76},
Bennett and Sharpley \cite{BSh}, Brudnyi and Kruglyak \cite{BK91}, Krein,
Petunin and Semenov \cite{KPS82}, Triebel \cite{Triebel}. One of the main
reasons for such fruitful applicability of interpolation theory is that, for
many couples $(X_{0},X_{1})$, we can effectively describe the class $%
Int(X_{0},X_{1})$ of all interpolation spaces. In most of the known cases of
couples $\left( X_{0},X_{1}\right) $ for which this is possible, this
description is formulated by using the Peetre $K$-functional, which plays an
extremely important role in the theory.

\vskip0.2cm

Since the first example of such a couple was obtained by Calder\'{o}n \cite{CAL}
and Mityagin \cite{Mit}, for those couples the terminology \textit{Calder\'{o}n couple} or \textit{Calder\'{o}n-Mityagin couple} is usually used. They proved
independently that a Banach function space $X$ on an arbitrary underlying
measure space is an interpolation space with respect to the couple $\left(
L^{1},L^{\infty }\right) $ on that measure space if and only if the
following monotonicity property\footnote{%
Here we are using Calder\'{o}n's terminology. Mityagin formulates this result
somewhat differently.} holds: if $f\in X$, $g\in L^{1}+L^{\infty }$ and 
\begin{equation*}
\int\nolimits_{0}^{t}g^{\ast }\left( s\right) \,ds\leq
\int\nolimits_{0}^{t}f^{\ast }\left( s\right) \,ds,\;\;t>0
\end{equation*}%
(where $h^{\ast }$ denotes the nonincreasing left-continuous rearrangement
of $|h|$), then $g\in X$ and $\left\Vert g\right\Vert _{X}\leq \left\Vert
f\right\Vert _{X}.$ Peetre \cite{PeetreJ1963Nouv,PeetreJ1968brasilia} had
proved (cf.~also a similar result due independently to Oklander \cite%
{OklanderE1964}, and cf. also \cite[ pp.~158--159]{KreeP1968}) that the
functional $\int\nolimits_{0}^{t}f^{\ast }\left( s\right) \,ds$ is in fact
the $K$-functional of the function $f\in L^{1}+L^{\infty }$ for the couple $%
\left( L^{1},L^{\infty }\right) $. Thus, the results of \cite{CAL,Mit}
naturally suggested that analogous results for couples other than $\left(
L^{1},L^{\infty }\right) $ might be expressed in terms of the $K$-functional
for those couples. To mention at least some of the papers, devoted to
investigations of this kind, we refer to \cite{LSh}, ($\left(
L^{p},L^{\infty }\right) ,1<p<\infty $), \cite{SaSe71} (couples of weighted $%
L^{1}$-spaces), \cite{Dmi75} (relative interpolation of couples $\left(
L^{1}(w_{0}),L^{1}(w_{1})\right) $ and $\left( L^{1},L^{\infty }\right) $ ), 
\cite{PeetreJ1971x} (relative interpolation of an \textit{arbitrary } Banach
couple with a couple of weighted $L^{\infty }$-spaces) (see also \cite[%
p.~589, Theorem 4.4.16]{BK91} or \cite[p.~29, Theorem 4.1]{CwPe81} for this
result), \cite{SP78} (couples of weighted $L^{p}$-spaces, $1\leq p\leq
\infty $), and \cite{Kal93} (couples of rearrangement invariant spaces).

\vskip0.2cm

It is worth to note that, for an arbitrary Banach couple, the uniform $K$%
-monotone interpolation spaces, which are closely related to the 
Calder\'{o}n-Mityagin property (all definitions see in Section \ref{Prel} below) can
also be described in a more concrete way. This important fact is due to
Brudnyi and Kruglyak \cite{BK91} and follows from their proof (see \cite[%
pp.~503-504]{BK91}) of a conjecture due to S. G. Krein \cite{DiKrOv77}. One
of its consequences is that, for every Banach couple $(X_{0},X_{1})$ with
the Calder\'{o}n-Mityagin property, the family $Int(X_{0},X_{1})$ of all its
interpolation Banach spaces can be parameterized by the set of Banach
function lattices on $(0,\infty )$, so-called $K$\textit{-method parameters}%
. The proof of this fundamental fact is based on using the so-called $K$%
-divisibility property of Banach couples.

\vskip0.2cm

Almost all of the results listed above were obtained for couples of \textit{%
Banach}{} spaces (as rare exceptions, we mention that in \cite{SP78} Sparr
was in fact also able to treat couples of weighted $L^{p}$ spaces for $p\in
(0,\infty )$ under suitable hypotheses, and then Cwikel \cite{Cwikel1}
considered the couple $\left( l^{p},l^{\infty }\right) $ also for $p$ in
this extended range). At the same time, new questions have recently arisen
(see, for instance, \cite{LSZ17a}, \cite{CN17}) that require analogous
results for more general situations, say, for quasi-Banach spaces. The
extension of the basic concepts and constructions of interpolation theory to
the latter setting (and even for couples of quasi-normed Abelian groups) was
initiated long ago by Peetre and Sparr in \cite{PS}. Much more recently in 
\cite{AsCwNi21} it was shown that the couple $\left(l^{p},l^{q}\right) $ has
the Calder\'{o}n-Mityagin property if and only if $q\geq 1$. In contrast to
that, from Theorem 1.1 in \cite{CSZ} (see also \cite[Remark~4.7]{AsCwNi21})
it follows that the couple $(L^{p},L^{q})$ of measurable functions on the
semi-axis $(0,\infty )$ with the Lebesgue measure has the Calder\'{o}n-Mityagin
property for all $0<p<q\leq \infty $. Observe also that some other partial
results for the latter couples in the non-Banach case were obtained somewhat
earlier in \cite{CN17,LSZ17a,Ast-20,Cad}.

\vskip0.2cm

The main aim of this paper is to develop a general approach which allows to
extend the basics of Brudnyi-Kruglyak theory to the realm of quasi-Banach
lattices. First, we obtain a refined version of the $K$-divisibility
property for Banach lattice couples and then prove the appropriate version
of this property for $L$-convex quasi-Banach lattice couples. Then we show
that all $K$-monotone quasi-Banach lattices with respect to a $L$-convex
quasi-Banach lattice couple have in fact a stronger property of the
so-called $K(p,q)$-monotonicity for some $0<q\leq p\leq 1$, which allows us
to get their description by the real $K$-method. Moreover, if additionally a 
$K$-monotone quasi-Banach lattice with respect to such a couple is $L$%
-convex, then the corresponding parameter can be selected to be also $L$%
-convex.

\vskip0.2cm

Let us describe now the main results of the paper in more detail. In Section %
\ref{Prel}, we give preliminaries with basic definitions and notation. So,
we address some properties of quasi-Banach spaces and lattices which relate
to their convexity, highlighting the class of $L$-convex lattices as it was
defined by Kalton in \cite{Kal84}. In the next section we collect some
auxiliary results, which are apparently known to some extent.

\vskip0.2cm

The central result of Section \ref{Divisibility for Banach} is Theorem \ref%
{K-Divisibility for Lattices}, showing that the important $K$-divisibility
property for Banach lattice couples has additional useful lattice features.
Observe that in the more restricted case of function lattices on a $\sigma $%
-finite measure space a close result can be found in \cite[Theorem~4.1]%
{CwNi03}. In the next section, the last result is used to establish the fact
that every $L$-convex quasi-Banach lattice couple possesses the so-called $p$%
-$K$-divisibility property for some $p>0$, which is a natural substitute of
the usual $K$-divisibility for Banach couples (see Theorem \ref%
{K-divisibility}).

\vskip0.2cm

In Section~\ref{K( p,q)-monotone spaces}, we introduce and study the
property of the so-called $K(p,q)$-monotonicity ($0<q\leq p\leq 1)$ of
intermediate quasi-Banach spaces with respect to quasi-Banach couples. It is
stronger than the usual $K$-monotonicity and closely connected with some
convexity properties of spaces (see Proposition \ref{pr6.8}).

\vskip0.2cm

In Section \ref{description of K-monotone lattices}, we focus on the central
problem considered in this paper: When a quasi-Banach $K$-monotone space $X$
with respect to a quasi-Banach couple $\overline{X}=(X_{0},X_{1})$ is
described by the real $K$-method, i.e., can be represented as $\overline{X}_{E:K}$, where $E$ is a quasi-Banach function lattice on $(0,\infty )$? In
Theorem \ref{th6.2}, we prove that the answer is positive whenever $%
\overline{X}$ is a $L$-convex quasi-Banach lattice couple. Moreover, if $X$
is $L$-convex, parameter $E$ can be chosen to be $L$-convex (see Corollary %
\ref{cor0}). By using the $K$ $(p,q)$-monotonicity property, we obtain here
also a new description of the set of all $K$-monotone quasi-Banach spaces
with respect to an arbitrary $p$-convex quasi-Banach lattice couple.

\vskip0.2cm

Finally, in Section \ref{Applications}, we continue earlier investigations
of the papers \cite{CSZ} and \cite{AsCwNi21}, stating some consequences of
the above results both for sequence $l^{p}$- and function $L^{p}$-spaces, $%
0<p<\infty $, focusing on the non-Banach case. It is worth to note that the
interpolation properties of the couples $\left( L^{p},L^{q}\right) $ (when the underlying measure space is $(0,\infty)$ with the Lebesgue measure) and $\left( l^{p},l^{q}\right) $ are rather different. While the first one is a
uniform Calder\'{o}n-Mityagin couple for all parameters $0<p<q\leq \infty $ \cite%
[Theorem~1.1]{CSZ} (see also \cite[Remark~4.7]{AsCwNi21}), the second
possesses the latter property if and only if $q\geq 1$ \cite[Corollary~5.4]%
{AsCwNi21}. This fact causes a difference in the description of $K$-monotone
quasi-Banach lattices with respect to these couples by the real $K$-method
in Theorems \ref{Th2} and \ref{Th4}. 

\vskip0.2cm

\section{Preliminaries}

\label{Prel}

\subsection{Quasi-Banach spaces and lattices.}

\label{Prel4}

Recall that a (real) quasi-Banach space $X$ is a complete real vector space
whose topology is given by a quasi-norm $x\mapsto \|x\|$ which satisfy the
conditions: $\|x\|>0$ if $x\ne 0$, $\|\alpha x\|=|\alpha|\|x\|$, $\alpha\in%
\mathbb{R}$, $x\in X$, and $\|x_1+x_2\|_X\le C(\|x_1\|+\|x_2\|)$ for some $%
C>0$ and all $x_1,x_2\in X$.

\vskip0.2cm

Let $0<p\leq 1$. A quasi-Banach space $X$ is said to be \textit{$p$-normable}
if for some constant $B>0$ and any $x_{k}\in X$, $k=1,\dots ,n$, we have 
\begin{equation}
\Big\|\sum_{k=1}^{n}x_{k}\Big\|_{X}\leq B\Big(\sum_{k=1}^{n}\Vert x_{k}\Vert
_{X}^{p}\Big)^{1/p}.  \label{equ0}
\end{equation}%
By the Aoki-Rolewicz theorem (see e.g. \cite[Lemma~3.10.1]{BL76}), a
quasi-Banach space $X$ is $p$-normable whenever the constant $C$ in the
quasi-triangle inequality $\|x_1+x_2\|_X\le C(\|x_1\|+\|x_2\|)$ is equal to $2^{1/p-1}$. One can easily re-norm $X$ so that in %
\eqref{equ0} $B=1$.

\vskip0.2cm

Let $X$ be a quasi-Banach space and additionally is a vector lattice such
that $\Vert x\Vert \leq \Vert y\Vert $ whenever $|x|\leq |y|$. Then we say
that $X$ is a quasi-Banach lattice (see \cite{LT79-II}, \cite{MeyNie91}).

\vskip0.2cm

We say that a quasi-Banach lattice $X$ satisfies an \textit{upper $p$%
-estimate} if for some constant $C$ and any pairwise disjoint $x_k\in X$, $%
k=1,\dots,n$ 
\begin{equation*}  \label{equ1}
\Big\|\sum_{k=1}^n x_k\Big\|_X\le C\Big(\sum_{k=1}^n \|x_k\|_X^p\Big)^{1/p}.
\end{equation*}

A quasi-Banach lattice $X$ is said to be \textit{(lattice) $(p,q)$-convex}, $0<q\leq p\leq \infty $, $q<\infty $, if for some constant $M$ and any $x_{k}\in X$, $k=1,\dots ,n$, we have 
\begin{equation}
\Big\|\Big(\sum_{k=1}^{n}|x_{k}|^{p}\Big)^{1/p}\Big\|_{X}\leq M\Big(%
\sum_{k=1}^{n}\Vert x_{k}\Vert _{X}^{q}\Big)^{1/q}  \label{equ2}
\end{equation}%
(with the usual modification when $p=\infty $). By $M^{\left( p,q\right)
}\left( X\right) $ we will denote the minimal $M$ satisfying \eqref{equ2}.
Observe that any element of the form $(\sum_{k=1}^{n}|x_{k}|^{p})^{1/p}$, $%
0<p<\infty $, can be defined by means of a "homogeneous functional calculus"
in the quasi-Banach setting exactly as in the case of Banach lattices (cf. 
\cite[pp.~40-41]{LT79-II}, \cite{Kal84}, \cite{CT-86}).

\vskip0.2cm

In particular, when $p=q>0$, we come to the well-known notion of a $p$%
-convex quasi-Banach lattice (see e.g. \cite{LT79-II} for $p\ge 1$ and \cite%
{Kal84} for $p>0$). We set $M^{\left( p\right)
}\left(X\right):=M^{\left(p,p\right) }\left(X\right)$.

It is clear that for each $0<p\leq 1$ $p$-convexity implies $p$-normability
and $p$-normability implies the existence of an upper $p$-estimate. On the
other hand, if a quasi-Banach lattice is $q$-normable, $0<q\leq 1$, then it
is $(1,q)$-convex (see also \cite[Proposition~1.3]{CT-86}).

\vskip0.2cm

According to \cite[p.~142]{Kal84}, a quasi-Banach lattice $X$ is said to be 
\textit{$L$-convex} if there is $0<\varepsilon <1$ so that if $0\leq x\in X$
with $\left\Vert x\right\Vert _{X}=1$ and $0\leq x_{i}\leq x$, $i=1,\dots ,n$%
, satisfy 
\begin{equation*}
\frac{1}{n}\sum_{i=1}^{n}x_{i}\geq (1-\varepsilon )x,
\end{equation*}%
then 
\begin{equation*}
\max_{i=1,\dots ,n}\left\Vert x_{i}\right\Vert _{X}\geq \varepsilon .
\end{equation*}%
Throughout the paper, we repeatedly use the fact that a quasi-Banach lattice 
$X$ is $L$-convex if and only if it is $r$-convex for some $r>0$ \cite[%
Theorem~2.2]{Kal84}. In particular, this implies that $L$-convexity is
preserved under all equivalent lattice quasi-Banach renormings.

\vskip0.2cm

In the case $1\leq q\leq p<\infty $ the concept of a $(p,q)$-convex Banach
lattice has been introduced by Maurey \cite{Maurey-74}, which showed that
this adds nothing to $p$-convexity, because any $(p,q)$-convex Banach
lattice is also $r$-convex for each $r<q$. In contrast to the Banach case,
there are $(p,q)$-convex quasi-Banach lattices $X$ for some $0<q\leq p$,
which are not $r$-convex for any $r<q$ and hence not $L$-convex. An example
of such a lattice (for $p=\infty $) are the spaces $L^{p}(\phi )$ defined by
Kalton \cite[Example~2.4]{Kal84}.

\vskip0.2cm

Next, we will often use also the so-called $p$-convexification procedure
which just an abstract description of the mapping $f\mapsto |f|^p\mathrm{sign%
}\,f$ from $L_r(\mu)$, $0<r<\infty$, to $L_{rp}(\mu)$. Note that in a
general lattice $X$ there is no meaning to the symbol $x^p$ that makes us to
introduce new algebraic operations in $X$ (see \cite[pp.~53-54]{LT79-II}).

Let $X$ be a quasi-Banach lattice with the algebraic operations denoted by $%
+ $ and $\cdot $ and let $p>0$. For every $x,y\in X$ and $\alpha \in \mathbb{%
R} $ we define 
\begin{equation*}
x\oplus y:=(x^{1/p}+y^{1/p})^{p}\;\;\mbox{and}\;\;\alpha \odot x:=\alpha
^{p}\cdot x,
\end{equation*}%
where $(x^{1/p}+y^{1/p})^{p}$ is the element in $X$ defined according to 
\cite[Theorem~1.d.1]{LT79-II} and $\alpha ^{p}$ is $|\alpha |^{p}\mathrm{sign%
}\,\alpha $. Then, the set $X$, endowed with the operations $\oplus $, $%
\odot $ and the same order as in $X$ is a vector lattice, which denoted by $%
X^{(p)}$. Moreover, $|||x|||_{X^{(p)}}:=\Vert x\Vert _{X}^{1/p}$ is a
lattice norm on $X^{(p)}$ and $(X^{(p)},|||\cdot|||_{X^{(p)}})$ is a
quasi-Banach lattice for every $0<p<\infty$ \cite[Proposition~1.2]{CT-86}.
One can easily check also that if $X$ is $r$-convex with the constant $%
M^{(r)}$, then $X^{(p)}$ is $pr$-convex with the same constant.

\vskip0.2cm

Note that if $X$ is a quasi-Banach function lattice, $X^{(p)}$ can be
identified with the space of functions $f$ such that $f^{p}:=|f|^{p}\mathrm{%
sign}\,f\in X$ equipped with the norm $|||f|||=\Vert \,|f|^{p}\,\Vert ^{1/p}$%
.

\vskip0.2cm

The simplest family of $p$-normable function spaces is formed by $L^{p}$%
-spaces, $0<p\le\infty $. As usual, when the underlying measure space is $%
(0,\infty)$ with the Lebesgue measure, they are equipped with the
quasi-norms 
\begin{equation*}
\Vert f\Vert_{{p}}:=\Big(\int_{0}^{\infty}|f(t)|^{p}\,dt\Big)^{1/p}\;\;%
\mbox{if}\;\;p<\infty,\;\;\mbox{and}\;\;\Vert f\Vert_{{\infty}}:=\mathrm{%
ess\, sup}_{t>0}|f(t)|.
\end{equation*}

Similarly, $l^{p}$-spaces, $0<p\le\infty$, of sequences $x=(x_k)_{k=1}^%
\infty $ are equipped with the quasi-norms 
\begin{equation*}
\Vert x\Vert _{l^{p}}:=\Big(\sum_{k=1}^{\infty }|x_{k}|^{p}\Big)^{1/p}\;\;%
\mbox{if}\;\;p<\infty,\;\;\mbox{and}\;\;\Vert x\Vert _{l^{\infty
}}:=\sup_{k=1,2,\dots }|x_{k}|.
\end{equation*}

For every $0<p\le\infty $ we have 
\begin{equation*}
\Vert f+g\Vert _{{p}}\leq \max (1,2^{(1-p)/p})\left( \Vert f\Vert _{{p}%
}+\Vert g\Vert _{{p}}\right),\;\;f,g\in L^{p},
\end{equation*}%
\begin{equation*}
\Vert x+y\Vert _{l^{p}}\leq \max (1,2^{(1-p)/p})\left( \Vert x\Vert
_{l^{p}}+\Vert y\Vert _{l^{p}}\right), x,y\in l^{p}
\end{equation*}%
(see e.g. \cite[Lemma~3.10.3]{BL76}). Moreover, both spaces $L^{p}$ and $%
l^{p}$ are $p$-convex for $0<p<\infty$ with constant $1$ and hence they are
Banach spaces if $1\le p\le\infty$, and $L$-convex quasi-Banach spaces if $%
0<p<1$.

\subsection{Interpolation of quasi-Banach spaces}

\label{Prel1}

Let us recall some basic constructions and definitions related to the
interpolation theory of operators. For more detailed information we refer to 
\cite{BL76,BK91,BSh,KPS82,Ovc84}.

\smallskip {} \vskip0.2cm

In this paper we are mainly concerned with interpolation within the class of
quasi-Banach spaces, while the linear bounded operators are considered as
the corresponding morphisms. A pair $\overline{X}=(X_{0},X_{1})$ of
quasi-Banach spaces is called a \textit{\ quasi-Banach couple} if $X_{0}$
and $X_{1}$ are both linearly and continuously embedded in some Hausdorff
linear topological space.

\vskip0.2cm

For each quasi-Banach couple $\overline{X}=(X_{0},X_{1})$ we define the 
\textit{\ intersection} $\,\Delta \left( \overline{X}\right) =X_{0}\cap
X_{1} $ and the \textit{sum} $\Sigma \left( \overline{X}\right) =X_{0}+X_{1}$
as the quasi-Banach spaces equipped with the quasi-norms 
\begin{equation*}
\Vert x\Vert _{\Delta \left( \overline{X}\right) }:=\max \left\{ \Vert
x\Vert _{X_{0}}\,,\,\Vert x\Vert _{X_{1}}\right\}
\end{equation*}%
and 
\begin{equation*}
\Vert x\Vert _{\Sigma \left( \overline{X}\right) }:=\inf \left\{ \Vert
x_{0}\Vert _{X_{0}}\,+\,\Vert x_{1}\Vert _{X_{1}}:\,x=x_{0}+x_{1},x_{i}\in
X_{i},i=0,1\right\} ,
\end{equation*}%
respectively.

\vskip0.2cm

A quasi-Banach space $X$ is called \textit{intermediate} with respect to a
quasi-Banach couple $\overline{X}$ if $\Delta \left( \overline{X}\right)
\subset X\subset \Sigma \left( \overline{X}\right) $ continuously. We will
say that $X$ is a \textit{normalized intermediate space} if both inclusions
have norm one. The set of intermediate spaces is denoted with $I\left( 
\overline{X}\right) $ or $I\left( X_{0},X_{1}\right) $.

\vskip0.2cm

If $\overline{X}=(X_{0},X_{1})$ and $\overline{Y}=(Y_{0},Y_{1})$ are
quasi-Banach couples, then we let $\mathfrak{L}\left( \overline{X},\overline{%
Y}\right) $ denote the space of all linear operators $T:\,\Sigma \left( 
\overline{X}\right) \rightarrow \Sigma \left( \overline{Y}\right) $ that are
bounded from $X_{i}$ in $Y_{i}$, $i=0,1$, equipped with the quasi-norm 
\begin{equation*}
{\Vert T\Vert }_{\mathfrak{L}\left( \overline{X},\overline{Y}\right)
}:=\max\limits_{i=0,1}{\ \Vert T\Vert }_{X_{i}\rightarrow Y_{i}}.
\end{equation*}%
In the case when $X_{i}=Y_{i}$, $i=0,1$, we simply write $\mathfrak{L}\left( 
\overline{X}\right) $ or $\mathfrak{L}(X_{0},X_{1})$.

\vskip0.2cm

Let $\overline{X}=(X_{0},X_{1}),\overline{Y}=\left( Y_{0},Y_{1}\right) $ be
two quasi-Banach couples and let $X\in I\left( \overline{X}\right)$, $Y\in I\left( \overline{Y}\right) $. Then, $X,\,Y$ are said to be \textit{relative interpolation spaces}
with respect to the couples $\overline{X},\overline{Y},$ if every operator $T%
{\in \mathfrak{\overline{L}}\left( \overline{X},\overline{Y}\right) }$ is
bounded from $X$ into $Y.$ If $\overline{X}=\overline{Y}$ and $X=Y,$ then $X$
is an \textit{interpolation space}\emph{\ }with respect to $\overline{X}.$
The collection of all such spaces will be denoted by $Int\left( \overline{X}%
\right) $ or $Int\left( X_{0},X_{1}\right) .$

\vskip0.2cm

By the above-mentioned Aoki-Rolewicz theorem (see also \cite[Lemma~3.10.1]%
{BL76}), every quasi-Banach space is a $F$-space (i.e., the topology in that
space is generated by a complete invariant metric). In particular, this
applies to the space ${\mathfrak{L}}\left( \overline{X},\overline{Y}\right)$
which is obviously a quasi-Banach space with respect to the quasi-norm $%
T\mapsto {\Vert T\Vert }_{\mathfrak{L}\left( \overline{X},\overline{Y}%
\right)}$ and also with respect to the quasi-norm $T\mapsto \max({\Vert
T\Vert }_{\mathfrak{L}\left( \overline{X},\overline{Y}\right)},{\Vert T\Vert 
}_{X\to Y})$ whenever the quasi-Banach spaces $X,\,Y$ are relative
interpolation spaces with respect to the quasi-Banach couples $\overline{X},%
\overline{Y}$. As is well known (see e.g. \cite[Theorem~2.2.15]{Rud}), the
Closed Graph Theorem (see e.g. \cite[Corollary 2.2.12]{Rud}) holds for $F$-spaces. Therefore, exactly the same reasoning as required for the Banach case (see Theorem 2.4.2 of \cite[p.~28]{BL76}) shows that in this case there exists a constant $C>0$ such
that for every $T\in \mathfrak{\overline{L}}\left( \overline{X},\overline{Y}\right) $ we have ${\Vert T\Vert }_{X\rightarrow Y}\leq C{\Vert T\Vert }_{\mathfrak{\overline{L}}\left( \overline{X},\overline{Y}\right) }.$ The least
constant $C$, satisfying the last inequality for all such $T$, is called the 
\textit{interpolation constant} of the spaces $X\,,Y$ with respect to the
couples $\overline{X},\overline{Y}$. If $C=1$, then $X,\,Y$ are called 
\textit{exact relative interpolation spaces} with respect to the couples $\overline{X},\overline{Y}$ (in the case when $\overline{X}=\overline{Y}$ and 
$X=Y, $ we say that $X$ is an \textit{exact interpolation space} with
respect to $\overline{X}$).

\vskip0.2cm

One of the most important ways of constructing interpolation spaces is based
on use of the \textit{Peetre {$K$}-functional}, which is defined for an
arbitrary quasi-Banach couple $\overline{X}=(X_{0},X_{1})$, for every $x\in
\Sigma \left( \overline{X}\right) $ and each $t>0$, as follows: 
\begin{equation}
{K}(t,x;\overline{X}):=\inf
\{||x_{0}||_{X_{0}}+t||x_{1}||_{X_{1}}:\,x=x_{0}+x_{1},x_{i}\in {X_{i},i=0,1}%
\}.
\end{equation}%
One can easily show that for a fixed $x\in \Sigma \left( \overline{X}\right) 
$ the function $t\mapsto ${$K$}$(t,x;\overline{X})$ is continuous,
nondecreasing, concave and non-negative on $\left( 0,\infty \right) $ \cite[%
Lemma~3.1.1]{BL76}.

\vskip0.2cm

Let $X,Y$ be intermediate quasi-Banach spaces with respect to quasi-Banach
couples $\overline{X}=\left( X_{0},X_{1}\right) $ and $\overline{Y}=\left(
Y_{0},Y_{1}\right) $, respectively. Then, $X,Y$ are said to be \textit{%
relative {$K$}-monotone spaces} with respect to these couples if whenever
elements $x\in X$ and $y\in \Sigma \left( \overline{Y}\right) $ satisfy 
\begin{equation*}
{K}\left( t,y;\overline{Y}\right) \leq {K}\left( t,x;\overline{X}\right) ,\;%
\text{for all\thinspace }\;t>0,
\end{equation*}%
it follows that $y\in Y$. If additionally $\left\Vert y\right\Vert _{Y}\leq
C\left\Vert x\right\Vert _{X}$, for a constant $C$ which does not depend on $%
x$ and $y$, then we say that $X,Y$ are \textit{uniform relative }$K-$\textit{%
monotone} spaces with respect to the couples $\overline{X}$ and $\overline{Y}
$. The infimum of all constants $C$ with this property is referred as the 
\textit{{$K$}-monotonicity constant }of $X$ and $Y$. Clearly, each pair of
uniform relative {$K$}${\mathcal{-}}$monotone spaces with respect to the
couples $\overline{X}$, $\overline{Y}$ are relative interpolation spaces
between $\overline{X}$ and $\overline{Y}.$ In the special case when $%
\overline{Y}=\overline{X}$ and $Y=X$, $X$ is called a \textit{(uniform) $K-$%
monotone space} with respect to $\overline{X}.$ The collection of all
uniformly $K-$monotone spaces with respect to the couple $\overline{X}$ is
denoted by $Int^{KM}\left( \overline{X}\right) $ or $Int^{KM}\left(
X_{0},X_{1}\right) .$

\vskip0.2cm

Assume that $E$ is a quasi-Banach function lattice (with respect to the
usual Lebesgue a.e. order) on $(0,\infty )$ and $w$ is a measurable function
on $(0,\infty )$. Then, $E(w)$ is the weighted quasi-Banach function lattice
with the norm $\Vert x\Vert _{E(w)}:=\Vert xw\Vert _{E}$. In what follows, $%
\overline{L^{\infty }}$ is the couple $\left( L^{\infty },L^{\infty
}(1/t)\right) $.

\vskip0.2cm

Let $\overline{X}$ be a quasi-Banach couple and let $E$ be an intermediate
quasi-Banach function lattice with respect to the couple $\overline{%
L^{\infty }}$. The space \textit{of the real $K$-method} (or the \textit{$K$%
-space}) $\overline{X}_{E:K}$ is defined as the set of all $x\in \Sigma
\left( \overline{X}\right) $ such that $K\left( \cdot ,x;\overline{X}\right)
\in E$. It is equipped with the quasi-norm $\left\Vert x\right\Vert
=\left\Vert K\left( \cdot ,x;\overline{X}\right) \right\Vert _{E}$ (see for details \cite{BK91} or \cite{Nil82}). Without
loss of generality, we can assume that the parameter $E$ is an interpolation
space with respect to the couple $\overline{L^{\infty }}$ \cite[%
Corollary~3.3.6]{BK91}. The collection of all $K$-spaces for the couple $%
\overline{X}$ will be denoted by $Int^{K}\left( \overline{X}\right) $ or
otherwise by $Int^{K}\left( X_{0},X_{1}\right) $.

In particular, if $E=L^{p}\left( t^{-\theta },\frac{dt}{t}\right) $, where $0<\theta <1,0<p\leq \infty $, we get the classical Lions-Peetre $K$-spaces $\overline{X}_{\theta ,p}$, i.e., 
\begin{equation*}
\Vert x\Vert _{\overline{X}_{\theta ,p}}:=\left( \int_{0}^{\infty }\left(
K\left( t,x;\overline{X}\right) t^{-\theta }\right) ^{p}\,\frac{dt}{t}%
\right) ^{1/p}
\end{equation*}%
(with usual modification if $p=\infty $), see \cite{BL76}.

\vskip0.2cm

If $\overline{X}=(X_{0},X_{1})$ is a quasi-Banach couple and $X\in I\left(\overline{X}\right)$, then the \textit{relative
closure} $X^{c}$ of $X$ consists of all $x\in \Sigma \left( \overline{X}%
\right) $ for which there exists a bounded sequence $\left\{ x_{n}\right\}
\subset X,$ converging to $x$ in $\Sigma \left( \overline{X}\right) .$ The
norm in $X^{c}$ is defined as the infimum of all bounds of such sequences in $%
X $. A space $X$ is (isometrically) \textit{relatively closed} in $\Sigma \left( \overline{X}\right) $ whenever $X^{c}=X$ isometrically. A quasi-Banach couple is said to
be \textit{mutually closed }if both spaces $X_{0}$ and $X_{1}$ are
relatively closed in $\Sigma \left( \overline{X}\right) $. Note that $%
K\left( t,x;\overline{X}\right) =K\left( t,x;\overline{X^{c}}\right) $ for
all $x\in \Sigma \left( \overline{X}\right) $ and $t>0$, where $\overline{%
X^{c}}=\left( X_{0}^{c},X_{1}^{c}\right) $ (see, for instance,\cite[Lemma 2]{Cwikel0} or \cite[formula (3.5), p.~385]{Ovc84}).

\vskip0.2cm

If $\overline{X}=\left( X_{0},X_{1}\right) $ and $\overline{Y}=\left(
Y_{0},Y_{1}\right) $ are two quasi-Banach couples, then the $\mathit{\ }$%
\textit{$K$}$\mathit{-orbit}$ ${{\mathrm{Orb}}}_{\overline{X}}^{K}\left( x;%
\overline{Y}\right) $ of an element $x\in \Sigma \left( \overline{X}\right) $%
, $x\neq 0$, in the couple $\overline{Y}=\left( Y_{0},Y_{1}\right) $ is the
space of all $y\in \Sigma \left( \overline{Y}\right) $ such that the
following quasi-norm 
\begin{equation*}
\left\Vert y\right\Vert _{{{\mathrm{Orb}}}^{K}(x)}:=\sup_{t>0}\frac{{K}%
\left( t,y;\overline{Y}\right) }{{K}\left( t,x;\overline{X}\right) }
\end{equation*}%
is finite. If $\overline{X}=\overline{Y},$ we write simply ${{\mathrm{Orb}}}%
^{K}\left( x;\overline{X}\right) $. One can easily check that each $K$-orbit
of an element $x\in \Sigma \left( \overline{X}\right) $, $x\neq 0$, is a
quasi-normed interpolation space between $X_{0}$ and $X_{1}$.

\vskip0.2cm

Two quasi-Banach couples $\overline{X}=\left( X_{0},X_{1}\right) $ and $%
\overline{Y}=\left( Y_{0},Y_{1}\right) $ are said to be \textit{relative
Calder\'{o}n-Mityagin couples} (or to have the \textit{relative 
Calder\'{o}n-Mityagin property}) if for every $x\in \Sigma \left( \overline{X}\right) $
and $y\in \Sigma \left( \overline{Y}\right) $ satisfying 
\begin{equation*}
K\left( t,y;\overline{Y}\right) \leq K\left( t,x;\overline{X}\right)
,\;\;t>0,
\end{equation*}%
there exists an operator $T\in \mathfrak{L}\left( \overline{X},\overline{Y}%
\right) $ such that $y=Tx$. If additionally we can choose $T\in \mathfrak{L}%
\left( \overline{X},\overline{Y}\right) $ so that $\Vert T\Vert _{\mathfrak{L%
}\left( \overline{X},\overline{Y}\right) }\leq C$, where $C$ is independent
of $x$ and $y$, then $\overline{X}$ and $\overline{Y}$ are called \textit{%
uniform relative Calder\'{o}n-Mityagin couples} (or we say that the above
couples have the \textit{uniform relative Calder\'{o}n-Mityagin property)}. In
the case when $\overline{X}=\overline{Y}$, any couple $\overline{X}$
satisfying the above property is called a \textit{(uniform) 
Calder\'{o}n-Mityagin couple}. See \cite{BK91} for more details.

\vskip0.2cm

The fact that $\overline{X}$ is a Calder\'{o}n-Mityagin couple obviously implies
that every interpolation space with respect to $\overline{X}$ is a $K%
\mathcal{-}$monotone space. Furthermore, if $\overline{X}$ is a uniform
Calder\'{o}n-Mityagin couple, then every interpolation space $X$ with respect to 
$\overline{X}$ is a uniform $K$-monotone space.

\subsection{Quasi-Banach lattice couples}

\label{Prel2}

Suppose $X_{0}$ and $X_{1}$ are two quasi-Banach lattices. We will say that $%
\overline{X}=(X_{0},X_{1})$ is a \textit{quasi-Banach lattice couple} if
there exists a Hausdorff topological vector lattice $\mathcal{H}$ such that
both $X_{0}$ and $X_{1}$ are embedded into $\mathcal{H}$ via a continuous,
interval preserving, lattice homeomorphism (see e.g. \cite{RaTr16}). Then,
one can easily check that the intersection $\Delta \left( \overline{X}%
\right) $ and the sum $\Sigma \left( \overline{X}\right) $ are quasi-Banach
lattices.

\vskip0.2cm

Suppose that a quasi-Banach lattice $X$ is intermediate with respect to a
quasi-Banach lattice couple $\overline{X}=(X_{0},X_{1})$ as a quasi-Banach
space. We will say that $X$ is an \textit{intermediate quasi-Banach lattice}
with respect to $\overline{X}$ if the canonical embeddings $I_{\Delta
}:\,\Delta \left( \overline{X}\right) \rightarrow X$ and $I_{\Sigma
}:\,X\rightarrow \Sigma \left( \overline{X}\right) $ are continuous,
interval preserving, lattice homeomorphisms (see \cite{BK91} and \cite%
{RaTr16}). \textbf{\ }By $I\left( \overline{X}\right) $ we will denote in
this case the set of all intermediate quasi-Banach lattices with respect to $%
\overline{X}$\textbf{.}

\vskip0.2cm

Observe that every $K$-monotone quasi-Banach lattice with respect to a
quasi-Banach lattice couple is, in fact, uniform $K$-monotone. This can be
proved by a simple modification of the arguments for the Banach lattice case
used in \cite[Theorem 6.1]{CwNi03}.

\vskip0.2cm

A quasi-Banach couple $\overline{X}=(X_{0},X_{1})$ is said to be \textit{$p$%
-normable} if both spaces $X_{0}$ and $X_{1}$ are $p$-normable. Similarly,
we will say that a quasi-Banach lattice couple $\overline{X}=(X_{0},X_{1})$
is \textit{$p$-convex (resp. $L$-convex)} if both spaces $X_{0}$ and $X_{1}$
are $p$-convex (resp. $L$-convex).

\vskip0.2cm

\subsection{The cone $Conv$ and the set of $K$-functionals of a couple.}

\label{Prel3}

Let $Conv$ denote the cone of all concave and nondecreasing functions $f$ defined on the half-line $[0,\infty )$ such that $f(0)=0$. Each $f\in Conv$ is a {\it quasi-concave} function, i.e., $f(t)$ is nondecreasing and  $t\mapsto f\left( t\right) /t$ is a nonincreasing function. It follows that $f\left( t\right) \leq \max \left(
1,t/s\right) f\left( s\right) $, and thus $Conv$ is a subset of the sum $%
\Sigma \left( \overline{L^{\infty }}\right) .$ Clearly, $\left\Vert
f\right\Vert _{\Sigma \left( \overline{L^{\infty }}\right) }=f\left(
1\right) .$ Moreover, one can easily check that, for each function $h\in
\Sigma (\overline{L^{\infty }})$, the $K$-functional $K\left( \cdot ,h;%
\overline{L^{\infty }}\right) $ is the least concave majorant of $|h|$ on $%
(0,\infty )$ \cite{Dm-74} (see also \cite[Proposition 3.1.17]{BK91}).

\vskip0.2cm

Note that, for every quasi-Banach couple $\overline{X}$ and any $x\in \Sigma
\left( \overline{X}\right) $, the $K$-functional $K\left( \cdot ,x;\overline{%
X}\right) $ belongs to the cone $Conv$. If conversely for every $f\in Conv$
there exists $x\in \Sigma \left( \overline{X}\right) $ such that $K\left(t,x;%
\overline{X}\right) \cong f(t)$, $t>0$, with equivalence constants
independent of $f$ (see \cite[Definition~4.4.8]{BK91}), then the
quasi-Banach couple $\overline{X}$ is called $Conv$-\textit{abundant}%
\footnote{%
This property is referred sometimes as the $K$-surjectivity of a couple, see
e.g. \cite[p. 217]{Nil83}.}. In particular, the couple $\overline{L^{\infty }%
}$ has the latter property.

\vskip0.2cm

By $Conv^{0}$ we will denote the subset of $Conv$ consisting of all elements 
$f$ with $\lim_{t\rightarrow 0}f\left( t\right)=\lim_{t\rightarrow \infty
}f\left( t\right) /t=0$.

\vskip0.2cm

The notation $A \preceq B$ means the existence of a positive constant $C$
with $A \leq C\cdot B$ for all applicable values of the arguments
(parameters) of the functions (expressions) $A$ and $B$. We will write $A
\cong B$ if $A \preceq B$ and $B \preceq A$.

\section{Auxiliary results}

\label{Aux}

\begin{proposition}
\label{renorming} Let $X$ be a $p$-convex quasi-Banach lattice, $p\in (0,1)$%
. Then the $1/p$-convexification $X^{\left( 1/p\right) }$ of $X$ has an
equivalent lattice norm and hence $X^{\left( 1/p\right) }$ is lattice
isomorphic to a Banach lattice.

In particular, if $X$ be a $L$-convex quasi-Banach lattice, then there is $%
p>0$ such that $X^{\left( p\right) }$ has an equivalent lattice norm.
\end{proposition}

\begin{proof}
Let $M^{(p)}$ be the $p$-convexity constant of $X$. Hence, if $x_{i}\in
X^{\left( 1/p\right) }$, $i=1,2\dots ,n$, then from definition of the $(1/p)$%
-convexification $X^{\left( 1/p\right) }$ (see Section \ref{Prel4}) it
follows 
\begin{equation*}
\left\Vert \sum_{i=1}^{n}\oplus x_{i}\right\Vert _{X^{\left( 1/p\right)
}}\leq (M^{(p)})^p\sum_{i=1}^{n}\left\Vert x_{i}\right\Vert
_{X^{\left(1/p\right) }}.
\end{equation*}

Now, a straightforward inspection shows that the functional 
\begin{equation*}
\left\Vert \left\Vert x\right\Vert \right\Vert :=\inf \left\{
\sum_{i=1}^{n}\left\Vert x_{i}\right\Vert _{X^{\left( 1/p\right)
}}:\left\vert x\right\vert \leq \sum_{i=1}^{n}\oplus \left\vert
x_{i}\right\vert ,x_{i}\in X^{\left( 1/p\right) }\right\}
\end{equation*}%
is a lattice norm on the space $X^{\left( 1/p\right) }$, which is equivalent
to the original quasi-norm.

The second assertion follows now immediately because every $L$-convex
quasi-Banach lattice is $p$-convex for some $p>0$ \cite[Theorem~2.2]{Kal84}.
\end{proof}

\begin{lemma}
Let $\overline{X}=\left( X_{0},X_{1}\right) $ be a $L$-convex quasi-Banach
lattice couple. Then, $\Delta \left( \overline{X}\right) $ and $\Sigma\left( 
\overline{X}\right) $ are $L$-convex quasi-Banach lattices.
\end{lemma}

\begin{proof}
We show first that $\Delta \left( \overline{X}\right) $ is $L$-convex.
Suppose that $\varepsilon_i>0$, $i=0,1$, is the constant appearing in the
definition of the $L$-convexity of the space $X_i$, $i=0,1$. We put $%
\varepsilon:= \min_{i=0,1}\varepsilon_i$.

Assume now that $0\le x\in \Delta\left( \overline{X}\right) $, $\left\Vert
x\right\Vert _{\Delta \left(\overline{X}\right) }=1$, $x_{i}\in \Delta
\left( \overline{X}\right)$ such that $0\leq x_{i}\leq x$ and $\left(
x_{1}+...+x_{n}\right) /n\geq (1- \varepsilon) x.$ Let say $\left\Vert
x\right\Vert _{X_{0}}=1.$ Since by the assumption and the choice of $%
\varepsilon $ 
\begin{equation*}
\max_{i=1,..n}\left\Vert x_{i}\right\Vert _{X_{0}}\geq \varepsilon_0\ge
\varepsilon,
\end{equation*}
we have 
\begin{equation*}
\max_{i=1,...n}\left\Vert x_{i}\right\Vert _{\Delta \left( \overline{X}%
\right) }\geq \max_{i=1,..n}\left\Vert x_{i}\right\Vert _{X_{0}}\geq
\varepsilon.
\end{equation*}%
Hence, $\Delta \left( \overline{X}\right) $ is $L$-convex.

To get the same result for $\Sigma \left( \overline{X}\right)$ it suffices
to observe that there exists $p>0$ such that both $X_{0}$ and $X_{1}$ are
p-convex, and then, by \cite[Proposition~3.2]{AsNi21}, we conclude that $%
\Sigma \left( \overline{X}\right)$ is $p$-convex as well. Thus, $\Sigma
\left( \overline{X}\right)$ is $L$-convex \cite[Theorem~2.2]{Kal84}.
\end{proof}

\begin{lemma}
\label{L5} Suppose that $\overline{X}=(X_{0},X_{1})$ is a $L-$convex
quasi-Banach lattice couple and $E\in Int\left( \overline{L^{\infty }}%
\right) $ is a $L$-convex lattice. Then, the space $\overline{X}_{E:K}$ is $%
L $-convex as well.
\end{lemma}

\begin{proof}
Let $x_{k}\in \overline{X}_{E:K}$, $k=1,2,\dots ,n$, be arbitrary. By \cite[%
Theorem~2.2]{Kal84}, there is $r>0$ such that the couple $\overline{X}$ and
the lattice $E$ are $r$-convex. Consequently, applying \cite[Proposition~3.2]%
{AsNi21}, we have 
\begin{equation*}
K\left( t,\left( \sum_{k=1}^{n}|x_{k}|^{r}\right) ^{1/r};\overline{X}\right)
\leq M^{\left( r\right) }\left( \Sigma \left( \overline{X}\right) \right)
\left( \sum_{k=1}^{n}K\left( t,|x_{k}|;\overline{X}\right) ^{r}\right)
^{1/r}.
\end{equation*}%
Hence, by \cite[Lemma~3.3]{AsNi21}, 
\begin{eqnarray*}
\left\Vert \left( \sum_{k=1}^{n}|x_{k}|^{r}\right) ^{1/r}\right\Vert _{%
\overline{X}_{E:K}} &\leq &M^{\left( r\right) }\left( \Sigma \left( 
\overline{X}\right) \right) \left\Vert \left( \sum_{k=1}^{n}K\left( \cdot
,|x_{k}|;\overline{X}\right) ^{r}\right) ^{1/r}\right\Vert _{E} \\
&\leq &M^{\left( r\right) }\left( \Sigma \left( \overline{X}\right) \right)
M^{\left( r\right) }\left( E\right) \left( \sum_{k=1}^{n}\left\Vert K\left(
\cdot ,|x_{k}|;\overline{X}\right) \right\Vert _{E}^{r}\right) ^{1/r} \\
&=&M^{\left( r\right) }\left( \Sigma \left( \overline{X}\right) \right)
M^{\left( r\right) }\left( E\right) \left( \sum_{k=1}^{n}\left\Vert
x_{k}\right\Vert _{\overline{X}_{E:K}}^{r}\right) ^{1/r}.
\end{eqnarray*}%
Thus, $\overline{X}_{E:K}$ is $r$-convex, whence it is $L$-convex.
\end{proof}

\section{$K$-Divisibility for Banach lattice couples}

\label{Divisibility for Banach}

Regarding the $K$-divisibility theorem for Banach couples, one of the
central result of the Brudnyi-Kruglyak interpolation theory, see \cite[%
Theorem 3.2.7]{BK91} and \cite{Cw84}.

\vskip0.2cm

Here, we will show that in the case of Banach lattice couples the $K$
-divisibility property has the additional useful feature that elements $%
x_{n} $, $n=1,2,\dots$, in a representation of a nonnegative element $x\in \Sigma \left( 
\overline{X}\right) $ can be selected also to be nonnegative. Note that in
the more restricted case of function lattices on a $\sigma $-finite measure
space a close result can be proved in \cite[Theorem~4.1]{CwNi03}.

\begin{theorem}
\label{K-Divisibility for Lattices} Suppose $\overline{X}=(X_{0},X_{1})$ is
a Banach lattice couple and $x\in \Sigma \left( \overline{X}\right) $, $%
x\geq 0.$ Assume that 
\begin{equation}  \label{the main ineq}
K\left( t,x;\overline{X}\right) \leq \sum_{n=1}^{\infty }\varphi
_{n}(t),\;\;t>0,
\end{equation}%
where $\varphi _{n}\in Conv$ such that $\sum_{n=1}^{\infty }\varphi \left(
1\right) <\infty $.

Then there exist $x_{n}\in \Sigma \left( \overline{X}\right) $, $x_{n}\geq 0$%
, satisfying 
\begin{equation*}
x=\sum_{n=1}^{\infty }x_{n}\;\;\mbox{(convergence in}\;\Sigma \left( 
\overline{X}\right) )\;\;\mbox{and}\;\;K\left( t,x_{n};\overline{X}\right)
\leq \gamma ^{\prime }\varphi _{n}(t),\;n\in \mathbb{N},\;t>0,
\end{equation*}%
where $\gamma ^{^{\prime }}$ is a universal constant.
\end{theorem}

We show first that the standard decomposition property in Banach lattices (see e.g. \cite[p.~2]{LT79-II}) can be easily extended to infinite sums.

\begin{lemma}
\label{inf decomposition} Let $X$ be a Banach lattice and let $y,y_{n}\in X$, $n=1,2,\dots$, be non-negative elements such that $\sum_{n=1}^\infty
y_{n}\in X$ and $y\leq \sum_{n=1}^\infty y_{n}.$ Then there exists $z_{n}\in
X$ with $0\leq z_{n}\leq y_{n}$, $n=1,2,\dots$, and $y=\sum_{n=1}^\infty
z_{n}.$
\end{lemma}

\begin{proof}
Let $N$ be a fixed positive integer. By recursion, by \cite[p.~2]{LT79-II}, we can select $z_{n}\in X$, $n=1,2,\dots,N$ and $z_{N}\in X$ such that $0\leq z_{n}\leq y_{n}$, $%
n=1,2,\dots,N$, $0\leq z_{N}\leq \sum_{n=N+1}^{\infty }y_{n}$ and $y=\sum_{n=1}^{N}z_{n}+z_{N}.$ Since $z_{N}\rightarrow 0$ in $X$, it follows
that $y=\sum_{n=1}^\infty z_{n}$, and the desired result follows.
\end{proof}

\begin{proof}[Proof of Theorem \protect\ref{K-Divisibility for Lattices}]
Without loss of generality, we can assume that $\overline{X}$ is a mutually closed couple. Next, we use some arguments similar to those from the proof of \cite[Theorem 4.1, p.~44]{CwNi03}.

Let $x\in \Sigma \left( \overline{X}\right) $ be such that $x\geq 0$ and let 
$\varepsilon >0.$ Then, by the strong form of the fundamental lemma  \cite[Theorem 1.7]{CJM90} (for another version of this result see  \cite[Theorem 2.10]{KrSaSh05}), there exist $y_{n}\in \Delta \left( \overline{X}\right) $, $n\in \mathbb{Z}$, $y_{-\infty }\in X_{0}$, $y_{\infty }\in X_{1}$
such that $x=\sum_{n\in \mathbb{Z}}y_{n}+y_{-\infty }+y_{\infty }$ and 
\begin{equation}
\left\Vert y_{-\infty }\right\Vert _{X_{0}}+\sum_{n\in \mathbb{Z}}\min
\left( \left\Vert y_{n}\right\Vert _{X_{0}},t\left\Vert y_{n}\right\Vert
_{X_{1}}\right) +t\left\Vert y_{\infty }\right\Vert _{X_{1}}\leq \gamma
\left( 1+\varepsilon \right) K\left( t,x;\overline{X}\right),
\label{fund lemma}
\end{equation}
where $\gamma $ is the $K$-divisibility constant for the class of Banach
spaces. It is clear that $y=\sum_{n\in \mathbb{Z}}\left\vert
y_{n}\right\vert +\left\vert y_{-\infty }\right\vert +\left\vert y_{\infty
}\right\vert $ is an absolutely convergent series in $\Sigma \left( 
\overline{X}\right) $. Since $x\leq y,$ it follows from Lemma \ref{inf
decomposition} that there exist non-negative elements $z_{n},z_{-\infty
},z_{\infty }\in X$ such that $z_{n}\leq \left\vert y_{n}\right\vert
,z_{-\infty }\leq \left\vert y_{-\infty }\right\vert ,z_{\infty }\leq
\left\vert y_{\infty }\right\vert $ and $x=z_{-\infty }+\sum_{n\in \mathbb{Z}%
}z_{n}+z_{\infty }.$ Therefore, in view of \eqref{fund lemma}, we get 
\begin{equation}
\left\Vert z_{-\infty }\right\Vert _{X_{0}}+\sum_{n\in \mathbb{Z}}\min
\left( \left\Vert z_{n}\right\Vert _{X_{0}},t\left\Vert z_{n}\right\Vert
_{X_{1}}\right) +t\left\Vert z_{\infty }\right\Vert _{X_{1}}\leq \gamma
\left( 1+\varepsilon \right) K\left( t,x;\overline{X}\right) ,
\label{fund lemma1}
\end{equation}

Let $\mathcal{A}$ be the set of all $n\in \mathbb{Z}$, for which $z_{n}\neq
0 $, and let $\widetilde{\mathcal{A}}$ be obtained by the addition to $%
\mathcal{A}$ two atoms $-\infty $ and $\infty $, i.e., $\widetilde{\mathcal{A%
}}=\mathcal{A}\cup \{-\infty \}\cup \{\infty \}.$ Consider the linear space $%
l^{1}\left( w_{0}\right) $ of all sequences $a=\left( a_{n}\right) _{n\in 
\widetilde{\mathcal{A}}}$ of real numbers such that $a_{\infty }=0$, which
is equipped with the norm 
\begin{equation*}
\Vert a\Vert _{l^{1}\left( w_{0}\right) }:=\left\vert a_{-\infty
}\right\vert \left\Vert z_{-\infty }\right\Vert _{X_{0}}+\sum_{n\in \mathcal{%
A}}\left\vert a_{n}\right\vert \left\Vert z_{n}\right\Vert _{X_{0}}.
\end{equation*}%
In other words, the latter space is the usual weighted $l^{1}$-space over
the set $\widetilde{\mathcal{A}}$ with the counting measure such that the
corresponding weight $w_{0}$ may assume the value $+\infty $ on some set of positive measure. Namely, the weight $w_{0}$ is defined by $w_{0}(n)=\left\Vert
z_{n}\right\Vert _{X_{0}}$ if $n\in \mathcal{A} $, $w_{0}(-\infty
)=\left\Vert z_{-\infty }\right\Vert _{X_{0}}$ and $w_{0}(\infty )=\infty $.

Similarly, let $l^{1}\left( w_{1}\right) $ denote the weighted $l^{1}$-space
over the set $\widetilde{\mathcal{A}}$, where $w_{1}(n)=\left\Vert
z_{n}\right\Vert _{X_{1}}$ if $n\in \mathcal{A}$, $w_{1}(-\infty )=\infty $
and $w_{1}(\infty )=\left\Vert z_{\infty }\right\Vert _{X_{1}}$. Thus, $%
l^{1}\left( w_{1}\right) $ consists of all sequences $a=\left( a_{n}\right)
_{n\in \widetilde{\mathcal{A}}}$ such that $a_{-\infty }=0$ and 
\begin{equation*}
\Vert a\Vert _{l^{1}\left( w_{1}\right) }:=\left\vert a_{\infty }\right\vert
\left\Vert z_{\infty }\right\Vert _{X_{1}}+\sum_{n\in \mathcal{A}}\left\vert
a_{n}\right\vert \left\Vert z_{n}\right\Vert _{X_{1}}<\infty .
\end{equation*}

Let $\overline{A}:=\left(l^{1}\left( w_{0}\right),l^{1}\left( w_{1}\right)\right)$, $e_{n}$, $n\in \widetilde{\mathcal{A}}$, be the standard basis vectors, and let $e$ be the element of $\Sigma \left( \overline{A}\right) $ with all
entries equal to $1$. Define on $\Sigma \left( \overline{A}\right) $ the
linear operator $T$ by $T\left( e_{n}\right) =z_{n}$, $n\in \mathcal{A}$, $%
T\left( e_{-\infty }\right) =z_{-\infty }$ and $T\left( e_{\infty }\right)
=z_{\infty }.$ Then, $T$ is a positive operator such that $T:\overline{A}%
\rightarrow \overline{X}$ and $T\left( e\right) =x.$ Moreover, from %
\eqref{fund lemma1} and the well-known equivalence formula for the $K$-functional for a couple of weighted $l^{1}$-spaces (see e.g. \cite[%
Lemma~3.4.1]{Ovc84}) it follows that 
\begin{equation}
K\left( t,e;\overline{A}\right) \cong \left\Vert z_{-\infty }\right\Vert
_{X_{0}}+\sum_{n\in \mathcal{A}}\min \left( \left\Vert z_{n}\right\Vert
_{X_{0}},t\left\Vert z_{n}\right\Vert _{X_{1}}\right) +t\left\Vert z_{\infty
}\right\Vert _{X_{1}}\preceq K\left( t,x;\overline{X}\right)
\label{fund lemma2}
\end{equation}

Assume now that $\varphi \in Conv.$ By the representation theorem for
concave functions \cite[Lemma~5.4.3]{BL76}, we can write 
\begin{equation}
\varphi \left( t\right) =\alpha +\beta t+\varphi _{0}\left( t\right) ,
\label{conc1}
\end{equation}%
where $\varphi _{0}\in Conv^{0}$, that is, $\varphi _{0}$ is a concave
nondecreasing function satisfying the condition: $\lim_{t\rightarrow
0,\infty }\min \left( 1,1/t\right) \varphi _{0}\left( t\right) =0$.
According to \cite[p.~398]{Ovc84}, one can select a sequence $b^{\prime }\in
\Sigma( \overline{l^{1}}):= l^{1}+l^{1}(2^{-k})$, where $l^{1}=l^{1}(\mathbb{Z})$, such that $b^{\prime }\geq 0$ and 
\begin{equation}
K( t,b^{\prime };\overline{l^{1}}) \cong \varphi _{0}\left( t\right)
,\;\;t>0.  \label{conc2}
\end{equation}

Let $\overline{B}$ be the couple obtained from the couple $\overline{l^{1}}$
by the addition to $\mathbb{Z}$ two atoms $-\infty $ and $\infty $,
similarly as in the construction of the couple $\overline{A}$ above. More
precisely, $\overline{B}=(l^{1}(v_{0}),l^{1}(v_{1}))$, where $l^{1}(v_{0})$
and $l^{1}(v_{1})$ are weighted $l^{1}$-spaces over the set $\mathbb{Z}\cup
\{\pm \infty \}$ with the counting measure, where $v_{0}(n)=1$ if $n\in 
\mathbb{Z}\cup \{-\infty \}$, $v_{0}(\infty )=\infty $, $v_{1}(n)=2^{-n}$ if 
$n\in \mathbb{Z}$, $v_{1}(-\infty )=\infty $, $v_{1}(\infty )=1$.

Also, let $b=(b_{n})_{\mathbb{Z}\cup \{\pm \infty \}}$ be such that $%
b_{n}=b_{n}^{\prime }$ if $n\in \mathbb{Z}$, $b_{-\infty }=\alpha $ and $%
b_{\infty }=\beta $. Then, by the above-mentioned formula for the $K$%
-functional, we have 
\begin{equation*}
K\left( t,b;\overline{B}\right) \cong \alpha +t\beta +\sum_{n\in \mathbb{Z}%
}b_{n}\min \left( 1,t2^{-n}\right) \cong \alpha +t\beta +K(t,b^{\prime };%
\overline{l^{1}}).
\end{equation*}%
Combining this together with \eqref{conc1} and \eqref{conc2} yields $K\left(
t,{b}:\overline{B}\right) \cong \varphi \left( t\right) $, $t>0$. Hence,
from \eqref{fund lemma2} it follows that the inequality $K\left( t,x;%
\overline{X}\right) \preceq \varphi \left( t\right) $, $t>0$, implies that $%
K\left( t,e;\overline{A}\right) \preceq K\left( t,b;\overline{B}\right) $, $%
t>0$. Therefore, using the extended version of the Sadaev-Semenov theorem
(see \cite{SaSe71}) due to Cwikel in \cite{Cw84} (see also \cite[Remark~4.1]%
{CwKo02}), we can find a positive bounded linear operator \thinspace $S:\,%
\overline{B}\rightarrow \overline{A}$ such that $S\left( b\right) =e$.

Assume now that all conditions of the theorem to be hold. As we observed
above, for each ${n}\in \mathbb{N}$ one can select $b_{n}\in \Sigma \left( 
\overline{B}\right) $ such that $b_{n}\geq 0$ and $\varphi _{n}(t)\cong
K\left( t,b_{n};\overline{B}\right) $, $t>0$. Since 
\begin{equation*}
K\left( t,\sum_{n=1}^{\infty }b_{n};\overline{B}\right) =\sum_{n=1}^{\infty
}K\left( t,b_{n};\overline{B}\right) \cong \sum_{n=1}^{\infty }\varphi
_{n}(t),\;\;t>0,
\end{equation*}%
and, by condition, $\sum_{n=1}^{\infty }\varphi _{n}(1)<\infty $, then $%
b:=\sum_{n=1}^{\infty }b_{n}\in \Sigma \left( \overline{B}\right) $ and 
\begin{equation*}
K\left( t,b;\overline{B}\right) \cong \sum_{n=1}^{\infty }\varphi
_{n}(t),\;\;t>0.
\end{equation*}
Hence, from assumption \eqref{the main ineq} and estimate 
\eqref{fund
lemma2} it follows that 
\begin{equation*}
K\left( t,e;\overline{A}\right) \preceq K\left( t,b;\overline{B}\right)
,\;\;t>0.
\end{equation*}%
Consequently, by the above discussion, there exists a positive linear
operator $S$ bounded from the couple $\overline{B}$ into the couple $%
\overline{A}$ such that $S\left( b\right) =e.$ Then, if $T$ is the operator
constructed in the first part of the proof, we have $TS\left( b\right) =x.$
Moreover, since $TS$ is a positive operator, then $x_{n}:=TS\left(
b_{n}\right) \geq 0$, $n=1,2,\dots $. As a result, we have 
\begin{equation*}
\sum_{n=1}^{\infty }x_{n}=TS\left( \sum_{n=1}^{\infty }b_{n}\right)
=TS\left( b\right) =x
\end{equation*}%
and for all $n\in \mathbb{N}$ 
\begin{equation*}
K\left( t,x_{n};\overline{X}\right) \preceq K\left( t,b_{n};\overline{B}%
\right) \cong \varphi _{n}\left( t\right) ,\;\;t>0,
\end{equation*}
with some universal constant. This completes the proof.
\end{proof}

\section{$K$-divisibility property for $L$-convex quasi-Banach lattice
couples}

\label{divisibility property for quasi-Banach}

In this section, by using Theorem \ref{K-Divisibility for Lattices}, we
prove a version of the $K$-divisibility property for $L$-convex quasi-Banach
lattice couples.

\begin{theorem}
\label{K-divisibility} Let $\overline{X}=(X_{0},X_{1})$ be a $p$-convex
quasi-Banach lattice couple, where $p\in (0,1]$. Suppose that $x\in \Sigma
\left( \overline{X}\right) $ and 
\begin{equation}
K\left( t,x;\overline{X}\right) \leq \left( \sum_{i=1}^{\infty }\varphi
_{i}\left( t\right) ^{p}\right) ^{1/p},\;\;t>0,  \label{eq1}
\end{equation}%
with $\varphi _{i}\in Conv$, $i=1,2,\dots $, satisfying 
\begin{equation*}
\sum_{i=1}^{\infty }\varphi _{i}\left( 1\right) ^{p}<\infty .
\end{equation*}

Then, there exist elements $0\leq x_{i}\in \Sigma \left( \overline{X}\right) 
$ such that for all $i=1,2,\dots $ 
\begin{equation}
K\left( t,x_{i};\overline{X}\right) \leq \gamma ^{\prime \prime }\cdot
\varphi _{i}\left( t\right) ,\;\;t>0,  \label{eq3}
\end{equation}%
where the constant $\gamma ^{\prime \prime }$ depends only on $p$ and the $p$%
-convexity constant of the couple $\overline{X}$, and 
\begin{equation}
\left\vert x\right\vert =\sup_{n=1,2,\dots }\left(
\sum_{i=1}^{n}x_{i}^{p}\right) ^{1/p}\;\;\mbox{(in}\;\;\Sigma \left( 
\overline{X}\right) ),\;\;\mbox{with}\;\;\lim_{n,m\rightarrow \infty }\left(
\sum_{i=n}^{m}x_{i}^{p}\right) ^{1/p}=0.  \label{eq2}
\end{equation}

In particular, if $\overline{X}=(X_0,X_1)$ is a $L$-convex quasi-Banach
lattice couple, then there is $p>0$, for which this couple possesses the
above $p$-$K$-divisibility property.
\end{theorem}

\begin{proof}
First, one can easily check that the functions $\psi_i(t):=\varphi_i \left(
t^{1/p}\right) ^{p}$, $i=1,2,\dots$, are quasi-concave on $[0,\infty)$.
Hence, we have 
\begin{equation}
\label{quasi-concavity}
\frac12\tilde{\psi}_i(t)\le \psi_i(t)\le \tilde{\psi}_i(t),
\;\;i=1,2,\dots,\;\;\mbox{and}\;\;t>0,
\end{equation}
where $\tilde{\psi}_i$ is the least concave majorant of $\psi_i$, $i=1,2,\dots$ \cite[Theorem~II.1.1]{KPS82}.

Let $x\in \Sigma\left( \overline{X}\right) $ satisfy the conditions of the theorem. Again we can assume that $x\geq 0$ (see \cite[Lemma~3.3]{AsNi21}).
Then, rewriting \eqref{eq1} in the form 
\begin{equation*}
K\left( t^{1/p},x;\overline{X}\right) ^{p}\leq \sum_{i=1}^{\infty }\tilde{\psi}_{i}(t),\;\;t>0,
\end{equation*}%
by \cite[Proposition~3.4]{AsNi21}, we obtain 
\begin{equation}
K\left( t,x;\overline{X^{(1/p)}}\right) \leq 2^{1-p}\sum_{i=1}^{\infty
}\tilde{\psi}(t),\;\;t>0.  \label{p-conf}
\end{equation}

Next, according to Proposition \ref{renorming}, each of the spaces $%
X_{0}^{\left( 1/p\right) }$ and $X_{1}^{\left( 1/p\right) }$ can be equipped
with an equivalent lattice norm such that the new lattices $Y_{0}$ and $%
Y_{1} $ are lattice isomorphic to some Banach lattices. Moreover (see the
proof of Proposition \ref{renorming}), 
\begin{equation*}
\Vert x\Vert _{Y_{k}}\leq \Vert x\Vert _{X_{k}^{\left( 1/p\right) }}\leq
(M_{k}^{(p)})^{p}\Vert x\Vert _{Y_{k}},\;\;k=0,1,
\end{equation*}%
where $M_{k}^{(p)}$ is the $p$-convexity constant of the space $X_{k}$, $%
k=0,1$. Consequently, 
\begin{equation}
K\left( t,x;\overline{Y}\right) \leq K\left( t,x;\overline{X^{(1/p)}}\right)
\leq \max_{k=0,1}(M_{k}^{(p)})^{p}K\left( t,x;\overline{Y}\right) ,
\label{ineq for K-f}
\end{equation}%
whence, by \eqref{p-conf}, 
\begin{equation*}
K\left( t,x;\overline{Y}\right) \leq 2^{1-p}\sum_{i=1}^{\infty }\tilde{\psi}_{i}(t),\;\;t>0.
\end{equation*}%
Now, applying Theorem \ref{K-Divisibility for Lattices} to the Banach couple 
$\overline{Y}$, we find $0\leq x_{i}\in \Sigma (\overline{Y})=\Sigma (%
\overline{X^{\left( 1/p\right) }})=\Sigma \left( \overline{X}\right)
^{\left( 1/p\right) }$,  $i=1,2,\dots $ (see \cite[Proposition~3.4]{AsNi21}), such that 
\begin{equation*}
x=\sum_{i=1}^{\infty }\oplus x_{i}\;\;\mbox{(convergence in}\;\;\Sigma
\left( \overline{X}\right) ^{\left( 1/p\right) })
\end{equation*}%
and 
\begin{equation*}K(t,x_{i};\overline{Y})\leq 2^{1-p}\gamma ^{\prime }\tilde{\psi}_{i}(t).
\end{equation*}%
Hence, we immediately imply \eqref{eq2}. Moreover, using successively \cite[Proposition~3.4]{AsNi21}, inequality \eqref{ineq for K-f} (replacing $x$ with $x_{i}$), the latter inequality, and \eqref{quasi-concavity} with the definition of $\psi _{i}$, we
get for all $i=1,2,\dots $ and $t>0$
\begin{eqnarray*}
K(t^{1/p},x_{i};\overline{X})^{p} &\leq
&\max_{k=0,1}(M_{k}^{(p)})^{p}K\left( t,x_{i};\overline{X^{(1/p)}}\right)
\leq \max_{k=0,1}(M_{k}^{(p)})^{2p}K\left( t,x_{i};\overline{Y}\right) \\
&\leq &2^{1-p}\gamma ^{\prime }\max_{k=0,1}(M_{k}^{(p)})^{2p}\tilde{\psi}_{i}\left(t\right)\leq 2^{2-p}\gamma ^{\prime }\max_{k=0,1}(M_{k}^{(p)})^{2p}\varphi_{i}\left( t^{1/p}\right) ^{p}.
\end{eqnarray*}%
Therefore, we come to inequality \eqref{eq3} with the constant 
$$
\gamma^{\prime \prime }:=4^{1/p}(\gamma ^{\prime})^{1/p}\max_{k=0,1}(M_{k}^{(p)})^{2}.$$

Finally, if $\overline{X}=(X_0,X_1)$ is a couple of $L$-convex quasi-Banach
lattices, choosing $p>0$ so that both lattices $X_0$ and $X_1$ are $p$%
-convex, we get the last assertion of the theorem.
\end{proof}

\begin{remark}
\label{special1} From the proof of Theorem \ref{K-divisibility} it follows
that 
\begin{equation*}
|x|=\sum_{i=1}^{\infty }\oplus x_{i}\;\;\mbox{(convergence in}%
\;\;\Sigma\left( \overline{X}\right) ^{\left( 1/p\right) }).
\end{equation*}
\end{remark}

In the next section we introduce a concept, which will allow us to get a new
description of $K$-monotone spaces with respect to both Banach and $L$%
-convex quasi-Banach lattice couples.

\section{$K\left( p,q\right) $-monotone spaces}

\label{K( p,q)-monotone spaces}

For each $q>0$ and any quasi-Banach space $X$ by $l^{q}(X)$ we denote the set of all sequences $\left\{ x_{i}\right\} _{i=1}^{\infty }\subset X$ such that $\sum_{i=1}^\infty \|x_i\|^q<\infty$.

\begin{definition}
\label{def1} Let $0<q\leq p\leq 1$ and let $X$ be an intermediate quasi-Banach space with respect to a quasi-Banach couple $\overline{X}.$ We say that $X$ is a \textit{$K\left( p,q\right) $-monotone
space} with respect to $\overline{X}$ if the inequality 
\begin{equation}
K\left( t,x;\overline{X}\right) \leq \left( \sum_{i=1}^{\infty }K\left(
t,x_{i};\overline{X}\right) ^{p}\right) ^{1/p},\;\;t>0,  \label{equa1}
\end{equation}%
for some $\left\{ x_{i}\right\} _{i=1}^{\infty }\in l^{q}(X)$ and $x\in
\Sigma \left( \overline{X}\right) $ %
implies that $x\in X$ and 
\begin{equation*}
\left\Vert x\right\Vert _{X}\leq C\left( \sum_{i=1}^{\infty }\left\Vert
x_{i}\right\Vert _{X}^{q}\right) ^{1/q},
\end{equation*}%
with a constant $C>0$ independent of $\left\{ x_{i}\right\} _{i=1}^{\infty }$. Denote by $R_{\left( p,q\right) }\left( X\right) $ the minimal constant
satisfying the latter condition.

Similarly, $X$ is called a \textit{finitely } $K\left( p,q\right) $ \textit{%
monotone space }if the above condition holds uniformly for all finite
sequences $\left\{ x_{i}\right\} _{i=1}^{n}\subset X$. The corresponding
optimal constant (for all $n\in \mathbb{N}$) will be denoted by $R_{\left(
p,q\right) ,f}(X)$.

If $R_{\left( p,q\right) }\left( X\right) =1$ (resp. $R_{\left( p,q\right)
,f}\left( X\right) =1$), then $X$ will be called \textit{exact} (resp. 
\textit{finitely exact}) $K\left( p,q\right) $-monotone space with respect
to the couple $\overline{X}$. Moreover, $Int_{\left( p,q\right) }^{KM}\left( 
\overline{X}\right) $ (resp. $Int_{\left( p,q\right) ,f}^{KM}\left( 
\overline{X}\right) $) is the collection of all $K\left( p,q\right) $%
-monotone (resp. finitely $K\left( p,q\right) $-monotone) spaces with
respect to $\overline{X}$.
\end{definition}

The next result is immediate.

\begin{proposition}
\label{pr6.1} Every finitely $K\left( p,q\right)$-monotone space with
respect to a quasi-Banach couple $\overline{X}$ is uniformly $K$-monotone
with respect to $\overline{X}$.
\end{proposition}

\begin{proposition}
\label{pr6.1a} If $X$ is a finitely $K\left( p,q\right)$-monotone space with
respect to a $p$-normable quasi-Banach couple $\overline{X}$, then $X$ is $q$%
-normable.
\end{proposition}

\begin{proof}
It suffices to apply Definition \ref{def1} to a sum $x=\sum_{i=1}^{n}x_{i}$, 
$x_i\in X$, $i=1,\dots,n$.
\end{proof}

\begin{proposition}
\label{pr6.2} Let $X$ be a mutually closed intermediate quasi-Banach space
with respect to a quasi-Banach couple $\overline{X}$. If $X$ is finitely $%
K\left( p,q\right) $-monotone space then $X$ is $K\left( p,q\right) $%
-monotone.
\end{proposition}

\begin{proof}
Assume that $\left\{ x_{i}\right\} _{i=1}^{\infty }\in l^{q}(X)$ and $x\in
\Sigma \left( \overline{X}\right) $ satisfy \eqref{equa1}. Then, for any
fixed integer $N$ we have 
\begin{eqnarray*}
K\left( t,x;\overline{X}\right) &\leq &2^{1/p}\left( \sum_{i=1}^{N}K\left(
t,x_{i};\overline{X}\right) ^{p}\right) ^{1/p} \\
&+&2^{1/p}\left( \sum_{i=N+1}^{\infty }K\left( t,x_{i};\overline{X}\right)
^{p}\right) ^{1/p},\;\;t>0.
\end{eqnarray*}%
By using the finite version of $K$-divisibility \cite[Theorem~3.2.12]{BK91}
for quasi-Banach couples, for any fixed positive integer $N$, we can find $%
y_{N}\in \Sigma \left( \overline{X}\right) $ such that 
\begin{eqnarray*}
K\left( t,y_{N};\overline{X}\right) &\leq &\gamma ^{\prime 1/p}\left(
\sum_{i=1}^{N}K\left( t,x_{i};\overline{X}\right) ^{p}\right) ^{1/p},\;\;t>0,
\\
K\left( t,x-y_{N};\overline{X}\right) &\leq &\gamma ^{\prime 1/p}\left(
\sum_{i=N+1}^{\infty }K\left( t,x_{i};\overline{X}\right) ^{p}\right)
^{1/p},\;\;t>0.
\end{eqnarray*}%
By the assumption, from the first inequality it follows that $y_{N}\in X$, $%
N=1,2,\dots $, and 
\begin{equation*}
\Vert y_{N}\Vert _{X}\leq \gamma ^{\prime 1/p}R_{(p,q),f}(X)\left(
\sum_{i=1}^{\infty }\left\Vert x_{i}\right\Vert _{X}^{q}\right)
^{1/q},\;\;N=1,2,\dots,
\end{equation*}
that is, the sequence $\{y_N\}_{N=1}^\infty$ is bounded in $X$. 
Furthermore, applying the second one for $t=1$, we get 
\begin{equation*}
\left\Vert x-y_{N}\right\Vert _{\Sigma \left( \overline{X}\right) }\leq
\gamma ^{\prime 1/p}\left( \sum_{i=N+1}^{\infty }\left\Vert x_{i}\right\Vert
_{\Sigma \left( \overline{X}\right) }^{p}\right) ^{1/p}\leq \gamma ^{\prime
1/p}M_{2}\left( \sum_{i=N+1}^{\infty }\left\Vert x_{i}\right\Vert
_{X}^{p}\right) ^{1/p},
\end{equation*}
where $M_{2}$ is the norm of the inclusion $X\subset \Sigma \left( \overline{%
X}\right) $. Consequently, since the right-hand side of this inequality
tends to zero as $N\rightarrow \infty $ and $X$ is mutually closed, we
conclude that $x\in X^{c}=X$ and 
\begin{equation*}
\Vert x\Vert _{X}\leq \gamma ^{\prime 1/p}R_{(p,q),f}(X)\left(
\sum_{i=1}^{\infty }\left\Vert x_{i}\right\Vert _{X}^{q}\right) ^{1/q}.
\end{equation*}
\end{proof}

Since $l^{s}\left( X\right) \subseteq l^{q}\left( X\right) $ if $0<s\leq q$, we get

\begin{proposition}
\label{pr6.3} Let $0<s\leq q$. Then, every $K\left( p,q\right)$-monotone
space with respect to a quasi-Banach couple $\overline{X}$ is also $K\left(
p,s\right)$-monotone with respect to $\overline{X}$.
\end{proposition}

\begin{proposition}
\label{pr6.5} Every finitely $K\left( p,q\right) $-monotone (resp. $K\left(
p,q\right) $-monotone) space with respect to a quasi-Banach couple $%
\overline{X}$ can be renormed to be exact finitely $K\left( p,q\right) $%
-monotone (resp. $K\left( p,q\right) $-monotone).
\end{proposition}

\begin{proof}
Since the proofs of both assertions are completely similar, we will deal
only with the first of them.

For each $x\in X$ we put 
\begin{equation*}
\psi \left( x\right) :=\inf \left( \sum_{i=1}^{n}\left\Vert x_{i}\right\Vert
_{X}^{q}\right) ^{1/q},
\end{equation*}%
where the infimum is taken over all finite sequences $\left\{ x_{i}\right\}
_{i=1}^{n}\subset X$ such that 
\begin{equation}
K\left( t,x;\overline{X}\right) \leq \left( \sum_{i=1}^{n}K\left( t,x_{i};%
\overline{X}\right) ^{p}\right) ^{1(p},\;\;t>0.  \label{renorm}
\end{equation}%
Then $R_{(p,q),f}(X)^{-1}\left\Vert x\right\Vert _{X}\leq \psi \left( x\right) \leq
\left\Vert x\right\Vert _{X}$ and hence the functional $x\mapsto \psi \left(
x\right) $ is a quasi-norm equivalent to the original quasi-norm $x\mapsto
\Vert x\Vert _{X}$.

Next, assume that we have \eqref{renorm} and $\left( \sum_{i=1}^{n}\psi
\left( x_{i}\right) ^{q}\right) ^{1/q}<\infty $. By the definition of $\psi
(x)$, for every $\varepsilon >0$ and each $i=1,2,\dots ,n$ we can select $%
m_{i}\in \mathbb{N}$ and $x_{i}^{j}\in X$, $j=1,2,\dots ,m_{i}$, so that 
\begin{equation*}
K\left( \cdot ,x_{i};\overline{X}\right) ^{p}\leq \sum_{j=1}^{m_{i}}K\left(
\cdot ,x_{i}^{j};\overline{X}\right) ^{p}\;\;\;\mbox{and}\;\;\psi \left(
x_{i}\right) ^{q}\geq \sum_{j=1}^{m_{i}}\left\Vert x_{i}^{j}\right\Vert
_{X}^{q}-\frac{\varepsilon }{2^{i}}.
\end{equation*}%
Hence, from \eqref{renorm} it follows 
\begin{equation*}
K\left( t,x;\overline{X}\right) \leq \left(
\sum_{i=1}^{n}\sum_{j=1}^{m_{i}}K\left( t,x_{i}^{j};\overline{X}\right)
^{p}\right) ^{1/p},\;\;t>0,
\end{equation*}%
and 
\begin{equation*}
\sum_{i=1}^{n}\psi \left( x_{i}\right) ^{q}\geq
\sum_{i=1}^{n}\sum_{j=1}^{m_{i}}\left\Vert x_{i}^{j}\right\Vert
^{q}-\varepsilon \geq \psi (x)^{q}-\varepsilon .
\end{equation*}%
Since $\varepsilon >0$ is arbitrary, these estimates imply that $X$ equipped
with the quasi-norm $\psi \left( x\right) $ is an exact finitely $K\left(
p,q\right) $-monotone space with respect to $\overline{X}$.
\end{proof}

\begin{proposition}
\label{pr6.6} Let $X$ be a $K\left( p,q\right) $-monotone quasi-Banach space
with respect to a quasi-Banach couple $\overline{X}=(X_{0},X_{1})$. Then $X$
can be equipped with an equivalent norm so that it becomes a normalized
intermediate and exact $K\left( p,q\right) $-monotone space with respect to $%
\overline{X}$ for some $s\in (0,q]$.
\end{proposition}

\begin{proof}
First, we select $s\in (0,q]$ so that the couple $\overline{X}$ and the
space $X$ are $s$-normable. Then, according to Proposition \ref{pr6.3}, $X$
is $K\left( p,s\right) $-monotone with respect to $\overline{X}$. Next, we
follow a reasoning due to Aronszajn-Gagliardo \cite{AG65}.

Since $X$ is an intermediate space with respect to $\overline{X}$, the
canonical inclusions $j_{X}:X\rightarrow \Sigma \left( \overline{X}\right) $%
, $i_{X}:\Delta \left( \overline{X}\right) \rightarrow X$ are bounded.
Consider the quasi-normed space $Y_{X}:=\left( X+\Delta \left( \overline{X}%
\right) \right) \cap \Sigma \left( \overline{X}\right) .$ Then, we have%
\textbf{.}%
\begin{equation*}
\left\Vert x\right\Vert _{\Sigma \left( \overline{X}\right) }\leq \left\Vert
x\right\Vert _{Y_{X}}\leq \max \left\{ \left\Vert x\right\Vert _{\Sigma
\left( \overline{X}\right) },\left\Vert x\right\Vert _{\Delta \left( 
\overline{X}\right) }\right\} =\left\Vert x\right\Vert _{\Delta \left( 
\overline{X}\right) }.
\end{equation*}%
Also, from the inequalities 
\begin{equation*}
\left\Vert x\right\Vert _{\Sigma \left( \overline{X}\right) }\leq \left\Vert
j_{X}\right\Vert _{X}\left\Vert x\right\Vert _{X}\;\;\mbox{and}%
\;\;\left\Vert x\right\Vert _{X+\Delta \left( \overline{X}\right) }\leq
\left\Vert x\right\Vert _{X},
\end{equation*}%
it follows 
\begin{equation*}
\left\Vert x\right\Vert _{Y_{X}}\leq \max \left( 1,\left\Vert
j_{X}\right\Vert \right) \left\Vert x\right\Vert _{X}.
\end{equation*}%
On the other hand, if $x=x_{0}+x_{1}\in X+\Delta \left( \overline{X}\right) $%
, $x_{0}\in X$, $x_{1}\in \Delta \left( \overline{X}\right) $, then 
\begin{equation*}
\left\Vert x\right\Vert _{X}\leq \left\Vert x_{0}\right\Vert _{X}+\left\Vert
i_{X}\right\Vert \left\Vert x_{1}\right\Vert _{\Delta \left( \overline{X}%
\right) },
\end{equation*}%
which implies that 
\begin{equation*}
\left\Vert x\right\Vert _{X}\leq \max \left( 1,\left\Vert i_{X}\right\Vert
\right) \left\Vert x\right\Vert _{X+\Delta \left( \overline{X}\right) }\leq
\max \left( 1,\left\Vert i_{X}\right\Vert \right) \left\Vert x\right\Vert
_{Y_{X}}.
\end{equation*}%
Thus, we can assume (after a suitable renorming) that $X$ is a normalized
intermediate space with respect to the couple $\overline{X}$.

Let now $\psi \left( x\right) $ denote the functional defined in the proof
of Proposition \ref{pr6.5} that provides $X$ with an equivalent norm such
that $X$ is an exact $K\left( p,s\right) $-monotone space with respect to $%
\overline{X}$. Then, by construction, 
\begin{equation*}
\psi \left( x\right) \leq \left\Vert x\right\Vert _{X}\leq \left\Vert
x\right\Vert _{\Delta \left( \overline{X}\right) },\;\;x\in \Delta \left( 
\overline{X}\right) .
\end{equation*}%
Furthermore, if $\left\{ x_{i}\right\} \in l^{s}(X)$ and $K\left( \cdot ,x:%
\overline{X}\right) \leq \left( \sum_{i=1}^{\infty }K\left( \cdot ,x_{i};%
\overline{X}\right) ^{p}\right) ^{1/p}$, we have 
\begin{equation*}
\left\Vert x\right\Vert _{\Sigma \left( \overline{X}\right) }\leq \left\Vert
x\right\Vert _{X}\leq \left( \sum_{i=1}^{\infty }\left\Vert x_{i}\right\Vert
_{X}^{s}\right) ^{1/s}.
\end{equation*}%
Hence, by the definition of $\psi (x)$, we get $\left\Vert x\right\Vert
_{\Sigma \left( \overline{X}\right) }\leq \psi \left( x\right) .$ Thus, $%
(X,\psi (\cdot ))$ is a normalized intermediate space with respect to the
couple $\overline{X}$ as well. This completes the proof.
\end{proof}

\begin{definition}
A normalized intermediate and exact $K\left( p,q\right)$-monotone space with
respect to a quasi-Banach couple $\overline{X}$ will be called a \textit{{%
normalized } $K\left( p,q\right)$-monotone space with respect to $\overline{X%
}$. }
\end{definition}

\begin{corollary}
Every $K\left( p,q\right)$-monotone space with respect to a quasi-Banach
couple $\overline{X}$ has an equivalent norm so that it is a normalized $%
K\left( p,s\right)$-monotone space for some $0<s\leq q$.
\end{corollary}

The next proposition establishes a connection between $K\left( p,q\right) $-monotonicity of a $K$-space with respect to a quasi-Banach couple and the
corresponding $K$-method parameter. We will say that a quasi-Banach lattice $E$ of
measurable functions on $(0,\infty )$ is $K\left( p,q\right) $-monotone if $%
E $ is $K\left( p,q\right) $-monotone with respect to the couple $\overline{%
L_{\infty }}$.

\begin{proposition}
\label{pr6.7} Let $E\in Int\left( \overline{L_\infty}\right) $ and let $%
\overline{X}$ be a quasi-Banach couple.

(i) If $E$ is $K\left( p,q\right)$-monotone, then so is $\overline{X}_{E:K}$
(with respect to $\overline{X}$).

(ii) If $\overline{X}_{E:K}$ is $K\left( p,q\right)$-monotone with respect
to $\overline{X}$ and $\overline{X}$ is $Conv$-abundant, then $E$ is $%
K\left( p,q\right)$-monotone.
\end{proposition}

\begin{proof}
(i) Assume that for any $x\in \Sigma \left( \overline{X}\right) $ and a
sequence $\{x_i\}\in l^{q}(\overline{X}_{E:K})$ we have 
\begin{equation*}
K\left(t,x;\overline{X}\right) \leq \left( \sum_{i=1}^{\infty
}K\left(t,x_{i};\overline{X}\right) ^{p}\right)^{1/p},\;\;t>0.
\end{equation*}
Then, since any function of the form $t\mapsto K\left(t,y;\overline{X}%
\right) $ belongs to $Conv$ (see Section \ref{Prel1} or \cite[Proposition
3.1.17]{BK91}), we have 
\begin{equation*}
K\left(t,K\left( \cdot ,x;\overline{X}\right);\overline{L_\infty}\right)\le
\left(\sum_{i=1}^{\infty }K\left(t,K\left( \cdot ,x_i:\overline{X}\right);\overline{L_\infty}\right)^p\right)^{1/p},\;\;t>0,
\end{equation*}
and 
\begin{equation*}
\left( \sum_{i=1}^{\infty }\left\Vert K\left( \cdot ,x_{i};\overline{X}%
\right)\right\Vert _{E}^{q}\right) ^{1/q}=\left( \sum_{i=1}^{\infty
}\left\Vert x_{i}\right\Vert _{\overline{X}_{E:K}}^{q}\right) ^{1/q}<\infty.
\end{equation*}
Consequently, by the assumption, it follows that $K\left( \cdot ,x;\overline{X}\right) \in E$, i.e., $x\in \overline{X}_{E:K}$. Moreover, 
\begin{equation*}
\|x\|_{\overline{X}_{E:K}}\le R_{(p,q)}(E)\left( \sum_{i=1}^{\infty
}\left\Vert x_{i}\right\Vert _{\overline{X}_{E:K}}^{q}\right) ^{1/q}.
\end{equation*}

(ii) Let $f\in \Sigma \left( \overline{L_\infty}\right) $ and $\{f_i\}\in
l^{q}(E)$ satisfy 
\begin{equation*}
K\left(t,f;\overline{L_\infty}\right) \leq \left(\sum_{i=1}^{\infty
}K\left(t,f_{i};\overline{L_\infty}\right) ^{p}\right)^{1/p},\;\;t>0.
\end{equation*}
Since $E\in Int\left( \overline{L_\infty}\right) $, we have 
\begin{equation*}
g_i(t):=K\left(t,f_{i};\overline{L_\infty}\right)\in E\;\;\mbox{and}%
\;\;\{g_i\}\in l^{q}(E).
\end{equation*}
Moreover, taking into account that $\overline{X}$ is $Conv$-abundant, we can
find $x_{i},x\in \Sigma\left(\overline{X}\right) $, $i=1,2,\dots$, such that 
\begin{equation*}
K\left(t,f;\overline{L_\infty}\right)\cong K\left(t,x;\overline{X}\right)\;\;%
\mbox{and}\;\; g_i(t)\cong K\left(t,x_i:\overline{X}\right),\;\;i=1,2,\dots,%
\;t>0,
\end{equation*}
with constants depending only on the couple $\overline{X}$. Thus, from the last inequality it follows 
\begin{equation*}
K\left(t,x;\overline{X}\right) \preceq \left(\sum_{i=1}^{\infty
}K\left(t,x_{i};\overline{X}\right) ^{p}\right)^{1/p},\;\;t>0.
\end{equation*}
Moreover, since $\{g_i\}\in l^{q}(E)$, we have $\{x_i\}\in l^{q}(\overline{X}%
_{E:K})$. Hence, by the assumption, $x\in \overline{X}_{E:K}$ and so $f\in
E. $ One can easily check also that the required estimate for the norm $\|f\|_E$ holds.
\end{proof}

The next result shows that the $K\left( p,q\right)$-monotonicity of a quasi-Banach lattice is closely connected with its convexity properties.

\begin{proposition}
\label{pr6.8} Let $0<q\le p\le 1$. If $\overline{X}=(X_0,X_1)$ is a $p$%
-convex quasi-Banach lattice couple and $X$ is an intermediate quasi-Banach
lattice with respect to $\overline{X}$, then the following conditions are
equivalent:

(a) $X$ is finitely $K\left( p,q\right)$-monotone;

(b) $X$ is $K$-monotone with respect to $\overline{X}$ and $\left(
p,q\right) $-convex;

(c) $X$ is $K\left( p,q\right)$-monotone.
\end{proposition}

\begin{proof}
$(a)\Longrightarrow (b)$. By Proposition \ref{pr6.1}, we need to prove only
that $X$ is $\left( p,q\right) $-convex.

Let $x_{i}\in X$, $i=1,\dots,n$. Since the couple $\overline{X}$ is $p$-convex, then $\Sigma\left(\overline{X}\right)$ (with the quasi-norm $%
x\mapsto K\left(t,x;\overline{X}\right)$) is a $p$-convex quasi-Banach
lattice \cite[Proposition~3.2]{AsNi21}. Consequently, for some $C>0$ 
\begin{equation*}
K\left(t,\left( \sum_{i=1}^{n}\left\vert x_{i}\right\vert ^{p}\right) ^{1/p}:%
\overline{X}\right) \leq C\left( \sum_{i=1}^{n}K\left(t,x_{i};\overline{X}%
\right) ^{p}\right) ^{1/p},\;\;t>0.
\end{equation*}%
Hence, by the assumption, 
\begin{equation*}
\left\Vert\left( \sum_{i=1}^{n}\left\vert x_{i}\right\vert ^{p}\right)
^{1/p}\right\Vert_X \leq C\cdot R_{(p,q),f}(X)\left(\sum_{i=1}^{n}\left\Vert
x_i\right\Vert _{X}^{q}\right) ^{1/q},
\end{equation*}
i.e., $X$ is $\left( p,q\right) $-convex.

$(b)\Longrightarrow (c)$. Observe first that, by \cite[Theorem 6.1]{CwNi03}
(see also Section \ref{Prel2}), we can assume that $X$ is a uniform $K$%
-monotone space with respect to the couple $\overline{X}$. Denote by $C_{X}$
the $K$-monotonicity constant of $X.$

Let $x\in \Sigma\left(\overline{X}\right)$ and $\{y_{i}\}_{i=1}^\infty\in
l^q(X)$ satisfy 
\begin{equation*}
K\left( t,x;\overline{X}\right) \leq \left( \sum_{i=1}^{\infty}K\left(
t,y_{i};\overline{X}\right) ^{p}\right) ^{1/p},\;\;t>0.
\end{equation*}
Observe that 
$$
\sum_{i=1}^{\infty}K\left(1,y_{i};\overline{X}\right) ^{p}\le M_2^p\sum_{i=1}^{n}\left\Vert y_i\right\Vert _{X}^{p}\le M_2^p\left(\sum_{i=1}^{n}\left\Vert y_i\right\Vert _{X}^{q}\right)^{p/q}<\infty,$$
where $M_2$ is the norm of the inclusion $X\subset \Sigma\left(\overline{X}\right)$. Therefore, applying Theorem \ref{K-divisibility}, we can select elements $%
x_{i}\in \Sigma \left( \overline{X}\right) $, $x_i\ge 0$, with 
\begin{equation}
K\left( t,x_{i};\overline{X}\right) \leq \gamma ^{\prime\prime }K\left(
t,y_{i};\overline{X}\right) ,\;\;i=1,2,\dots,\;\;t>0,  \label{equat100}
\end{equation}%
and 
\begin{equation*}
\left\vert x\right\vert =\sup_{n=1,2,\dots}\left(
\sum_{i=1}^{n}x_{i}^{p}\right) ^{1/p}\;\;(\;\mbox{in}\;\Sigma \left( 
\overline{X}\right)).
\end{equation*}%
Let us show that $x\in X$.

Since $X$ is a uniform $K$-monotone space with respect to the couple $%
\overline{X}$, from inequality \eqref{equat100} we infer that $x_{i}\in X$
and $\Vert x_{i}\Vert _{X}\leq C_{X}\cdot \gamma ^{\prime\prime }\Vert
y_{i}\Vert _{X}$, $i=1,2,\dots .$ Therefore, the hypothesis of the $(p,q)$%
-convexity of $X$ implies that for positive integers $m<n$ 
\begin{eqnarray*}
\left\|\sum_{i=m}^{n}\oplus x_{i}\right\|_{X^{\left( 1/p\right) }} &=&
\left\|\left( \sum_{i=m}^{n}x_{i}^{p}\right) ^{1/p}\right\|_X \leq
M^{(p,q)}(X)\left(\sum_{i=m}^{n}\left\Vert x_{i}\right\Vert _{X}^{q}\right)
^{1/q} \\
&\leq& C_{X\ }\gamma ^{\prime\prime }M^{(p,q)}(X)\left(
\sum_{i=m}^{\infty}\left\Vert y_{i}\right\Vert _{X}^{q}\right) ^{1/q}.
\end{eqnarray*}%
Consequently, the series $\sum_{i=1}^{\infty}\oplus x_{i}$ converges in $%
X^{\left( 1/p\right) }$ and for all $n=1,2,\dots$ 
\begin{equation*}
\left\|\left( \sum_{i=1}^{n}x_{i}^{p}\right) ^{1/p}\right\|_X \leq C_{X\
}\gamma ^{\prime\prime }M^{(p,q)}(X)\left( \sum_{i=1}^{\infty}\left\Vert
y_{i}\right\Vert _{X}^{q}\right) ^{1/q}.
\end{equation*}
Combining this inequality with Remark \ref{special1}, we conclude that $x\in
X $ and 
\begin{equation*}
\|x\|_X\le C_{X\ }\gamma ^{\prime\prime }M^{(p,q)}(X)\left(
\sum_{i=1}^{\infty}\left\Vert y_{i}\right\Vert _{X}^{q}\right) ^{1/q}.
\end{equation*}

Since implication $(c)\Longrightarrow (a)$ is obvious, the proof is completed.
\end{proof}

%

\begin{proposition}
\label{pr6.8a} Suppose $0<q\leq p\leq 1$, $E\in Int\left( \overline{%
L_{\infty }}\right) $ and $\overline{X}=(X_{0},X_{1})$ is a quasi-Banach couple.

(i) If $E$ is finitely $K\left( p,q\right) $-monotone, then $E$ is $\left(
p,q\right) $-convex.

(ii) If $E$ is $\left( p,q\right) $-convex, then the space $\overline{X}%
_{E:K}$ is finitely $K\left( p,q\right) $-monotone.
\end{proposition}

\begin{proof}
(i) Since $0<p\leq 1$, then the Banach couple $\overline{L}_{\infty }$ is $p$-convex. Therefore, the desired result follows from Proposition \ref{pr6.8}.

(ii) Assume that $x_{i}\in \overline{X}_{E:K}$, $i=1,\dots ,n$ and put $%
g(t):=\left( \sum_{i=1}^{n}K\left( t,x_{i};\overline{X}\right) ^{p}\right)
^{1/p}.$ If $K\left( t,x;\overline{X}\right) \leq g(t)$, $t>0$, we have 
\begin{equation*}
\left\Vert x\right\Vert _{\overline{X}_{E:K}}=\left\Vert K\left( \cdot ,x:%
\overline{X}\right) \right\Vert _{E}\leq \left\Vert g\right\Vert _{E}.
\end{equation*}%
Moreover, since $E$ is $\left( p,q\right) $-convex, we infer 
\begin{equation*}
\left\Vert g\right\Vert _{E}\leq M^{\left( p,q\right) }\left( E\right)
\left( \sum_{i=1}^{n}\left\Vert K\left( \cdot ,x_{i};\overline{X}\right)
\right\Vert _{E}^{q}\right) ^{1/q}=M^{\left( p,q\right) }\left( E\right)
\left( \sum_{i=1}^{n}\left\Vert x_{i}\right\Vert _{\overline{X}%
_{E:K}}^{q}\right) ^{1/q}.
\end{equation*}%
Combining these inequalities, we conclude that 
\begin{equation*}
\left\Vert x\right\Vert _{\overline{X}_{E:K}}\leq M^{\left( p,q\right)
}\left( E\right) \left( \sum_{i=1}^{n}\left\Vert x_{i}\right\Vert _{%
\overline{X}_{E:K}}^{q}\right) ^{1/q},
\end{equation*}%
and the proof completed.
\end{proof}

\section{A description of $K$-monotone lattices with respect to $L$-convex
quasi-Banach lattice couples by the real method.}

\label{description of K-monotone lattices}

One of the milestones of the interpolation theory developed by Brudnyi and
Kruglyak is the fundamental fact that every $K$-monotone Banach space with
respect to a Banach couple is described by the real $K$-method (see \cite[%
Theorem~4.1.11]{BK91}). In this section, we establish some results of such a
sort in the quasi-Banach setting provided if a lattice couple $\overline{X}$
possesses some non-trivial convexity. In addition, exploiting the above
notion of $K\left( p,q\right) $-monotonicity, we get a new description of $K$%
-monotone spaces with respect both Banach and $L$-convex quasi-Banach
couples.

\vskip0.2cm

Let $\overline{X}=(X_{0},X_{1})$ be a quasi-Banach couple and let $X$ be an
arbitrary intermediate quasi-Banach space with respect to $\overline{X}.$
Suppose $0<q\le p\le 1$. We begin with a construction of a suitable
parameter $\widehat{X}_{p,q}$ such that (under certain additional
conditions) $X=\overline{X}_{\widehat{X}_{p,q}:K}$. Our approach here is an
adoption of the arguments used in \cite[ Theorem~4.1.11]{BK91} (see also 
\cite[Theorem 4.1 and Corollary 4.3]{Nil83}).

\vskip0.2cm

We consider the linear space $\widehat{X}_{p,q}$ consisting of all
measurable functions $f$ on $[0,\infty )$, for which there exists a sequence 
$\left\{ x_{i}\right\} _{i=1}^{\infty }\in l^{q}\left( X\right) $ such that 
\begin{equation*}
\left\vert f(t)\right\vert \leq \left( \sum_{i=1}^{\infty }K\left( t,x_{i};%
\overline{X}\right) ^{p}\right) ^{1/p},\;\;t>0.
\end{equation*}%
Also, define the functional $f\mapsto \left\Vert f\right\Vert _{\widehat{X}%
_{p,q}}$ on $\widehat{X}_{p,q}$ by 
\begin{equation}
\left\Vert f\right\Vert _{\widehat{X}_{p,q}}:=\inf \left\{ \left(
\sum_{i=1}^{\infty }\left\Vert x_{i}\right\Vert _{X}^{q}\right)
^{1/q}\right\} ,  \label{p-norm}
\end{equation}%
where the infimum is taken over all sequences $\left\{ x_{i}\right\}
_{i=1}^{\infty }\in l^{q}\left( X\right) $ satisfying the preceding
inequality.

One can easily check that $\left\Vert f\right\Vert_{\widehat{X}_{p,q}}=0$ if
and only if $f=0$ and $\left\Vert \alpha f\right\Vert_{\widehat{X}%
_{p,q}}=|\alpha|\left\Vert f\right\Vert_{\widehat{X}_{p,q}}$ for each $%
\alpha\in\mathbb{R}$. In the next lemma we prove further properties of the
functional $f\mapsto \left\Vert f\right\Vert_{\widehat{X}_{p,q}}$.

\begin{lemma}
\label{properties of} Under the above conditions we have the following:

(i) $\widehat{X}_{p,q}$ is a $K\left( p,q\right)$-monotone quasi-Banach
lattice. Hence, $\widehat{X}_{p,q}$ is $(p,q)$-convex, and, in particular,
if $p=q$, it is $L$-convex.

(ii) $\widehat{X}_{p,q}$ is an exact interpolation space with respect to the
couple $\overline{L^{\infty }}.$
\end{lemma}

\begin{proof}
(i) First, we prove that $\widehat{X}_{p,q}$ is $K\left( p,q\right)$%
-monotone (with respect to $\overline{L_\infty}$). Let $f\in \Sigma\left(%
\overline{L_\infty}\right)$ and 
\begin{equation}  \label{special0}
K\left(t,f;\overline{L_\infty}\right)\leq \left( \sum_{i=1}^{\infty
}K\left(t,f_{i};\overline{L_\infty}\right) ^{p}\right) ^{1/p},\;\;t>0,
\end{equation}%
for some sequence $\left\{ f_{i}\right\} _{i=1}^{\infty }\in l^{q}\left( 
\widehat{X}_{p,q}\right) $. Then, by the definition of $\widehat{X}_{p,q}$,
for every $\varepsilon>0$ and for each $i=1,2,\dots$ there exists a sequence $%
\left\{ x_{i}^{k}\right\} _{k=1}^{\infty }\in l^{q}\left( X\right) $ such
that 
\begin{equation}  \label{special5}
\left\vert f_{i}(t)\right\vert \leq \left( \sum_{i=1}^{\infty }K\left(
t,x_{i}^{k};\overline{X}\right) ^{p}\right) ^{1/p},\;\;t>0,
\end{equation}
and 
\begin{equation}  \label{special2}
\|f_i\|_{\widehat{X}_{p,q}}^q\ge \sum_{k=1}^{\infty }\Vert x_{i}^{k}\Vert
_{X}^{q}-2^{-i-1}\varepsilon.
\end{equation}

Suppose that $0<p\leq 1$ and functions $\phi_i(t)\in Conv$, $i=1,2,\dots$, satisfy the condition $\sum_{i=1}^{\infty }\phi_i(1)^{p}<\infty$. Then, by the reverse triangle inequality for the $l^p$-norm, one can easily check that the function $\phi$ defined by 
\begin{equation*}
\phi(t)=\left( \sum_{i=1}^{\infty }\phi_i(t)^{p}\right) ^{1/p},\;\;t>0,
\end{equation*}%
belongs to the cone $Conv$ as well. Moreover, recall that, for
each $h\in \Sigma (\overline{L^{\infty }})$, the function $t\mapsto K\left( t,h;\overline{L^{\infty }}\right) $ is the least concave majorant of $|h|$ (see \cite{Dm-74} or \cite[Proposition~3.1.17]{BK91}). Taking into account these facts, we conclude that from \eqref{special5} it follows  
\begin{equation*}  
K\left(t,f_i;\overline{L_\infty}\right)\leq \left( \sum_{i=1}^{\infty
}K\left(t,x_{i}^{k};\overline{X}\right) ^{p}\right) ^{1/p},\;\;t>0.
\end{equation*}
Combining this together with \eqref{special0}, we get 
\begin{equation*}
\left\vert f(t)\right\vert \leq K\left(t,f;\overline{L_\infty}\right)\leq %
\Big(\sum_{k=1}^{\infty}\sum_{i=1}^{\infty }K\left( t,x_{i}^{k};\overline{X}%
\right)^{p}\Big)^{1/p},\;\;t>0.
\end{equation*}%
Moreover, from \eqref{special2} it follows that 
\begin{equation*}
\sum_{i=1}^{\infty }\sum_{k=1}^{\infty }\Vert x_{i}^{k}\Vert _{X}^{q}\le
\sum_{i=1}^{\infty }\|f_i\|_{\widehat{X}_{p,q}}^q+\varepsilon<\infty,
\end{equation*}
which implies that $\left\{ x_{i}^{k}\right\} _{i,k=1}^{\infty }\in
l^{q}\left( X\right) $. Therefore, 
\begin{equation*}
\|f\|_{\widehat{X}_{p,q}}\leq \Big(\sum_{k=1}^{\infty }\sum_{i=1}^{\infty
}\Vert x_{i}^{k}\Vert _{X}^{q}\Big)^{1/q}\le \Big(\sum_{i=1}^{\infty
}\|f_i\|_{\widehat{X}_{p,q}}^q+\varepsilon\Big)^{1/q}.
\end{equation*}%
Since $\varepsilon>0$ is arbitrary, this implies that $\widehat{X}_{p,q}$ is $%
K\left( p,q\right)$-monotone.

Next, from Proposition \ref{pr6.8a}(i) it follows that $\widehat{X}_{p,q}$
is $(p,q)$-convex, and if $p=q$, it is $L$-convex. In addition, since $%
0<p\leq 1$, we infer that the functional $f\mapsto \left\Vert f\right\Vert _{%
\widehat{X}_{p,q}}^{q}$ satisfies the triangle inequality
\begin{equation}  \label{p-q convexity}
\Big\|\sum_{k=1}^{\infty }f_k\Big\|_{\widehat{X}_{p,q}}\le \Big(\sum_{k=1}^{\infty }\|f_k\|_{\widehat{X}_{p,q}}^q\Big)^{1/q},
\end{equation}
and hence $\widehat{X}_{p,q}$ is a $q$-normable function space on $[0,\infty )$.
Consequently, to prove the completeness of $\widehat{X}_{p,q}$ it suffices
to prove that the condition 
\begin{equation}  \label{eq:compl}
\sum_{k=1}^{\infty }\Vert f_{k}\Vert _{\widehat{X}_{p,q}}^{q}<\infty.
\end{equation}%
implies convergence of the series $\sum_{k=1}^{\infty }f_{k}$ in this space.

From \eqref{eq:compl} it follows that for every $k=1,2,\dots$ there is a
sequence $\left\{ x_{i}^{(k)}\right\} _{i=1}^{\infty }\in l^{q}\left(
X\right) $, $k=1,2,\dots$, such that 
\begin{equation*}
\left\vert f_k(t)\right\vert \leq \left( \sum_{i=1}^{\infty }K\left(
t,x_{i}^{(k)};\overline{X}\right) ^{p}\right) ^{1/p},\;\;t>0,
\end{equation*}%
and 
\begin{equation*}
\sum_{k=1}^{\infty }\sum_{i=1}^{\infty }\left\Vert x_{i}^{(k)}\right\Vert
_{X}^{q}<\infty.
\end{equation*}
Moreover, since $\widehat{X}_{p,q}$ is algebraically and continuously
embedded into the space of all measurable functions on $(0,\infty )$ with
the topology of convergence in the Lebesgue measure on sets of finite
measure, combining \eqref{eq:compl} and \eqref{p-q convexity}, we infer that 
$f(t):=\sum_{k=1}^{\infty }f_{k}(t)$ is an a.e. finite measurable function
on $(0,\infty)$. In addition, from the above estimates and the fact that $%
p\le 1$ it follows for each $n=1,2,\dots$ 
\begin{equation*}
\sum_{k=n}^{\infty}|f_{k}(t)|\le \Big(\sum_{k=n}^{\infty}|f_{k}(t)|^{p}\Big)%
^{1/p}\le \Big(\sum_{k=n}^{\infty}\sum_{i=1}^{\infty }K\left( t,x_{i}^{(k)};%
\overline{X}\right)^{p}\Big)^{1/p},\;\;t>0,
\end{equation*}
whence 
\begin{equation*}
\Big\|\sum_{k=n}^{\infty}|f_{k}|\Big\|_{\widehat{X}_{p,q}}\le
\sum_{k=n}^{\infty }\sum_{i=1}^{\infty }\left\Vert x_{i}^{(k)}\right\Vert
_{X}^{q}.
\end{equation*}
Since the right-hand side of the last inequality tends to zero as $%
n\to\infty $, we conclude that the series $\sum_{k=1}^{\infty }f_{k}$
converges in the space $\widehat{X}_{p,q}$. Therefore, this space is
complete.

(ii) Let us check that $\widehat{X}_{p,q}$ is an intermediate space with
respect to the couple $\overline{L^{\infty }}.$ Show first that $\Delta
\left( \overline{L^{\infty }}\right) \subset \widehat{X}_{p,q}$.

Suppose $f\in \Delta \left( \overline{L^{\infty }}\right) $ and $\Vert
f\Vert _{\Delta \left( \overline{L^{\infty }}\right) }=1$. Then, we have 
\begin{equation*}
|f(t)|\leq \min (1,t),\;\;t>0.
\end{equation*}%
Fix any $x_{0}\in \Delta \left( \overline{X}\right) $ with $\Vert x_{0}\Vert
_{\Delta \left( \overline{X}\right) }=1$. From the definition of the $K$%
-functional if follows that 
\begin{equation*}
K(t,x_{0};\overline{X})\geq \min (1,t)K(1,x_{0};\overline{X}),\;\;t>0.
\end{equation*}%
Hence, denoting $c:=K(1,x_{0};\overline{X})$, we obtain 
\begin{equation*}
|f(t)|\leq \min (1,t)\leq c^{-1}K(t,x_{0};\overline{X}),\;\;t>0.
\end{equation*}%
Since $x_{0}\in X$, by the definition of $\widehat{X}_{p,q}$, we have $f\in 
\widehat{X}_{p,q}$ and 
\begin{equation*}
\Vert f\Vert _{\widehat{X}_{p,q}}\leq c^{-1}\Vert x_{0}\Vert _{X}\leq
c^{-1}M_{1},
\end{equation*}%
where $M_{1}$ is the norm of the inclusion of $\Delta \left( \overline{X}%
\right) $ into $X$. Then, using the simple homogeneity argument, we conclude
that $\Delta \left( \overline{L^{\infty }}\right) \subset \widehat{X}_{p,q}$.

Next, we show that $\widehat{X}_{p,q}\subset \Sigma \left( \overline{%
L^{\infty }}\right) $. Take $f\in \widehat{X}_{p,q}$ and select $\left\{
x_{i}\right\} _{i=1}^{\infty }\in l^{q}\left( X\right) $ such that 
\begin{equation}
\left\vert f(t)\right\vert \leq \left( \sum_{i=1}^{\infty }K\left( t,x_{i};%
\overline{X}\right) ^{p}\right) ^{1/p},\;\;t>0.  \label{eq4}
\end{equation}%
Combining this together with the inequality 
\begin{equation*}
\min \left( 1,1/t\right) K\left( t,x_{i};\overline{X}\right) \leq \left\Vert
x_{i}\right\Vert _{\Sigma \left( \overline{X}\right) }\leq M_{2}\left\Vert
x_{i}\right\Vert _{X},\;\;i=1,2,\dots \;\;\mbox{and}\;\;t>0,
\end{equation*}%
where $M_{2}$ is the norm of the inclusion $X\subset \Sigma \left( \overline{%
X}\right) $, we infer 
\begin{equation*}
\min \left( 1,1/t\right) \left\vert f(t)\right\vert \leq M_{2}\left(
\sum_{i=1}^{\infty }\left\Vert x_{i}\right\Vert _{X}^{q}\right)
^{1/q}<\infty .
\end{equation*}%
Hence, $f\in \Sigma \left( \overline{L^{\infty }}\right) $ and $\Vert f\Vert
_{\Sigma \left( \overline{L^{\infty }}\right) }\leq M_{2}\Vert f\Vert _{%
\widehat{X}_{p,q}}$.

Prove now that $\widehat{X}_{p,q}$ is a K-monotone space with respect to the
couple $\overline{L^{\infty }}$ with constant $1$. To this end, assume that $f\in \widehat{X}_{p,q}$, $g\in \Sigma (\overline{L^{\infty }})$ and 
\begin{equation}
K\left( t,g;\overline{L^{\infty }}\right) \leq K\left( t,f;\overline{%
L^{\infty }}\right) ,\;\;t>0.  \label{eq5}
\end{equation}%
Suppose $f$ satisfies \eqref{eq4} for a sequence $\left\{ x_{i}\right\}
_{i=1}^{\infty }\in l^{q}\left( X\right) $. Arguing in the same way as above, we see that from inequality \eqref{eq4} it follows  
\begin{equation*}
K\left( t,f;\overline{L^{\infty }}\right) \leq \left( \sum_{i=1}^{\infty
}K\left( t,x_{i};\overline{X}\right) ^{p}\right) ^{1/p},\;\;t>0.
\end{equation*}%
Combining this inequality with \eqref{eq5}, we infer 
\begin{equation*}
|g(t)|\leq K\left( t,g;\overline{L^{\infty }}\right) \leq \left(
\sum_{i=1}^{\infty }K\left( t,x_{i};\overline{X}\right) ^{p}\right)
^{1/p},\;\;t>0.
\end{equation*}%
Therefore, $g\in \widehat{X}_{p,q}$ and $\Vert g\Vert _{\widehat{X}%
_{p,q}}\leq \Vert f\Vert _{\widehat{X}_{p,q}}$. Thus, the desired claim is
proved. Since every K-monotone space with respect to the couple $\overline{%
L^{\infty }}$ with constant $1$ has the exact interpolation property with
respect to $\overline{L^{\infty }}$, the proof is complete.
\end{proof}

Now, we are ready to prove the main results of this section. Let us begin
with the case of Banach couples.

\begin{theorem}
\label{th6.1} Let $\overline{X}$ be a Banach couple and $X$ be a $K$%
-monotone quasi-Banach space with respect to $\overline{X}$. Then, $X$ is $%
K\left( 1,q\right) $-monotone for some $q\in (0,1]$. Moreover, we have 
\begin{equation*}
Int^{KM}\left( \overline{X}\right) =Int^{K}\left( \overline{X}\right)
=\bigcup_{0<q\leq 1}Int_{\left( 1,q\right) }^{KM}\left( \overline{X}\right) .
\end{equation*}
\end{theorem}

\begin{proof}
Recall first that $X$ is $q$-normable for some $0<q\le 1$ (see Section \ref{Prel4}). Then, if $X$ is a $K$-monotone quasi-Banach space with respect to a Banach couple $\overline{X}$, we have $X=\overline{X}_{E:K}$, where $E\in Int\left( \overline{L}_\infty\right) $ is 
$K\left( 1,q\right)$-monotone (see \cite[ Theorem~4.1.11]{BK91} for $q=1$ and \cite[Theorem 4.1 and Corollary 4.3]{Nil83} for $0<q<1$). Applying Proposition \ref{pr6.7}(i), we see that $X$ is $K\left( 1,q\right)$-monotone as well, and the first assertion of the theorem follows. In addition, we have
\begin{equation*}
Int^{KM}\left( \overline{X}\right) =Int^{K}\left( \overline{X}\right)\;\;\mbox{and}\;\;Int^{KM}\left( \overline{X}\right)\subset\bigcup_{0<q\leq 1}Int_{\left( 1,q\right) }^{KM}\left( \overline{X}\right) .
\end{equation*}
Since from Proposition \ref{pr6.1} it follows that every $%
K\left( 1,q\right)$-monotone space is uniformly $K$-monotone, the theorem is proved.
\end{proof}

\begin{corollary}
If a Banach couple $\overline{X}$ has the Calder\'{o}n-Mityagin property, then 
\begin{equation*}
Int\left( \overline{X}\right) =Int^{KM}\left( \overline{X}\right)
=Int^{K}\left( \overline{X}\right) =\bigcup _{0<q\leq 1}Int_{\left(
1,q\right) }^{KM}\left( \overline{X}\right) .
\end{equation*}
\end{corollary}

Let us proceed with the quasi-Banach setting.

\begin{theorem}
\label{th6.3} (a) Let $0<q\leq p\leq 1$ and let $X$ be a $K\left( p,q\right) 
$-monotone quasi-Banach space with respect to a quasi-Banach couple $%
\overline{X}.$ Then, $X=\overline{X}_{E:K}$ (with equivalence of norms),
where $E\in Int\left( \overline{L_{\infty }}\right) $. Furthermore, if $X$
is a normalized $K\left( p,q\right) $-monotone space, then $X$ is isometric
to the above $K$-space. Thus, every $K\left( p,q\right) $-monotone space $X$
can be renormed so that $X$ becomes isometric to a $K$-space with respect to 
$\overline{X}.$

(b) Conversely, let $\overline{X}$ be a quasi-Banach couple, $X=\overline{X}%
_{E:K}$ and either $X$ (with respect to $\overline{X}$) or $E$ (with respect
to $\overline{L_\infty}$) be mutually closed. Then $X$ is $K\left(
p,q\right) $-monotone for some $0<q\leq p\leq 1.$
\end{theorem}

\begin{proof}
(a) Show that it suffices to take for $E$ the lattice $\widehat{X}_{p,q}$
for the given $0\le q\le p\le 1$ (see its construction in the beginning of
the section). Recall that, by Lemma \ref{properties of}, $\widehat{X}_{p,q}$
is a $K\left( p,q\right)$-monotone quasi-Banach lattice, which is an exact
interpolation space with respect to the couple $\overline{L^{\infty }}$.

Let $\widetilde{X}$ denote the space $\overline{X}_{\widehat{X}_{p,q}:K}$.
By construction, it follows immediately that $X\subseteq \widetilde{X}$ and $%
\Vert x\Vert _{\widetilde{X}}\leq \Vert x\Vert _{X}$.

Conversely, assume~$x\in \widetilde{X}$, i.e., $K\left( \cdot ,x;\overline{X}%
\right) \in \widehat{X}_{p,q}.$ This means that there exists a sequence $%
\{x_{i}\}\in l^{q}(X)$ such that 
\begin{equation*}
K\left( t,x;\overline{X}\right) \leq \left( \sum_{i=1}^{\infty }K\left(
t,x_{i};\overline{X}\right) ^{p}\right) ^{1/p},\;\;t>0.
\end{equation*}%
Then, since $X$ is a $K\left( p,q\right)$-monotone space with respect to $%
\overline{X}$, we get $x\in X$ and 
\begin{equation*}
\|x\|_X\le R_{(p,q)}(X) \left( \sum_{i=1}^{\infty }\|x_i\|_X^q\right) ^{1/q}.
\end{equation*}
Taking the infimum over all admissible sequences $\{x_{i}\}\in l^{q}(X)$
implies then that $\|x\|_X\le R_{(p,q)}(X)\|x\|_{\widetilde{X}}$. Thus, $X=%
\widetilde{X}$, and the first assertion of part (a) is proved. The second
one follows now from the definition of a normalized $K\left( p,q\right)$%
-monotone space and Proposition \ref{pr6.6}.

(b) Choosing $p,s\in (0,1]$ in an appropriate way, we can assume that the
couple $\overline{X}$ and the lattice $E$ are $p$- and $s$-normable,
respectively. Clearly, then $E$ is $\left(1,s\right) $-convex, and hence
from \cite[Proposition~1.3]{CT-86} it follows that $E$ is $\left(p,q\right) $%
-convex, where $q$ is defined by $1/q-1/p=1/s-1.$ Thus, according to
Proposition \ref{pr6.8a}(ii), $X$ is finitely $K\left( p,q\right) $-monotone.

Now, if $X$ is mutually closed with respect to $\overline{X}$, it remains to
apply Proposition \ref{pr6.2}. The desired result follows also in the case
when $E$ is mutually closed with respect to $\overline{L_{\infty }}$,
because the last condition assures that the $K$-space $X=\overline{X}_{E:K}$
has analogous property with respect to $\overline{X}$ (see the reasoning used in \cite[Theorem 3.20]{Nil82} or \cite[Lemma 3.6.8 ]{BK91} adopted suitably to the quasi-Banach case).
\end{proof}

Next, we extend the result of Theorem \ref{th6.1} to the case of $p$-convex quasi-Banach lattice couples with $0<p<1$.

\begin{theorem}
\label{th6.2} Let $\overline{X}$ be a $p$-convex quasi-Banach lattice
couple, where $0<p<1$, and let $X$ be a $K$-monotone quasi-Banach lattice
with respect to $\overline{X}$. Then $X$ is $K\left( p,q\right) $-monotone
for some $q\in (0,p]$ and hence $X\in Int^{K}\left( \overline{X}\right) $.
Moreover, we have 
\begin{equation*}
Int^{KM}\left( \overline{X}\right) =Int^{K}\left( \overline{X}\right)
=\bigcup_{0<q\leq p\leq 1}Int_{\left( p,q\right) }^{KM}\left( \overline{X}%
\right) .
\end{equation*}
\end{theorem}

\begin{proof}
To prove the claim, we reduce the situation to the Banach case and then
apply Theorem \ref{th6.1}.

Since $\overline{X}$ is $p$-convex, by Proposition \ref{renorming}, the
couple $\overline{X^{\left( 1/p\right) }}$ is lattice isomorphic to a Banach
lattice couple. Show that the space $X^{\left( 1/p\right) }$ is $K$-monotone
with respect to $\overline{X^{\left( 1/p\right)} }$. Indeed, assume that $%
x\in X^{\left( 1/p\right) }$, $y\in \Sigma(\overline{X^{\left(
1/p\right) }})$ and 
\begin{equation*}
K\left(t,y;\overline{X^{\left( 1/p\right) }}\right)\le K\left(t,x;\overline{%
X^{\left( 1/p\right) }}\right),\;\;t>0.
\end{equation*}
Then, by \cite[Proposition~3.4]{AsNi21}, we have 
\begin{equation*}
K\left(t,y;\overline{X}\right)\le CK\left(t,x;\overline{X}\right),\;\;t>0,
\end{equation*}
with a constant $C>0$ independent of $x$ and $y$. Since $X$ is a $K$%
-monotone quasi-Banach lattice with respect to $\overline{X}$ and $x\in X$,
it follows that $y\in X^{\left( 1/p\right) }$, and the claim follows.

By Theorem \ref{th6.1}, we obtain that $X^{\left( 1/p\right) }$ is $%
K\left(1,r\right)$-monotone with respect to the couple $\overline{%
X^{\left(1/p\right) }}$ for some $r\in (0,1]$. Let us check that $X$ is $%
K\left( p,rp\right)$-monotone with respect to the couple $\overline{X}$.

Let 
\begin{equation*}
K\left(t,x;\overline{X}\right) \leq \left( \sum_{i=1}^{\infty
}K\left(t,x_{i};\overline{X}\right) ^{p}\right) ^{1/p},\;\;t>0,
\end{equation*}
for some $\left\{ x_{i}\right\}_{i=1}^\infty\in l^{rp}(X)$ and $x\in \Sigma
\left( \overline{X}\right) $. Then 
\begin{equation*}
K\left(t^{1/p},x;\overline{X}\right)^p \leq \sum_{i=1}^{\infty
}K\left(t^{1/p},x_{i};\overline{X}\right) ^{p},\;\;t>0,
\end{equation*}
whence, by \cite[Proposition~3.4]{AsNi21}, 
\begin{equation*}
K\left(t,x;\overline{X^{\left( 1/p\right) }}\right) \leq C^{\prime
}\sum_{i=1}^{\infty }K\left(t,x_{i};\overline{X^{\left( 1/p\right) }}%
\right),\;\;t>0,
\end{equation*}
with a constant $C^{\prime }>0$ independent of $x$ and $x_i$. Since $%
X^{\left( 1/p\right) }$ is $K\left( 1,r\right)$-monotone with respect to $%
\overline{X^{\left( 1/p\right) }}$, this yields 
\begin{equation*}
\|x\|_{X^{\left( 1/p\right) }}\le C^{\prime }R_{(1,r)}(X^{\left( 1/p\right)
})\left(\sum_{i=1}^{\infty }\|x_i\|_{X^{\left( 1/p\right) }}^r\right)^{1/r},
\end{equation*}
or equivalently, 
\begin{equation*}
\|x\|_{X}\le (C^{\prime }R_{(1,r)}(X^{\left( 1/p\right)
}))^{1/p}\left(\sum_{i=1}^{\infty }\|x_i\|_{X}^{rp}\right)^{1/(rp)}.
\end{equation*}
Thus, $X$ is $K\left( p,rp\right)$-monotone with respect to the couple $%
\overline{X}$, and the proof of the first assertion of the theorem is
completed. The second assertion is an immediate consequence of the first one
and Theorem \ref{th6.3}.
\end{proof}

\begin{corollary}
If $0<p<1$ and a $p$-convex quasi-Banach lattice couple $\overline{X}$ has
the Calder\'{o}n-Mityagin property, then 
\begin{equation*}
Int\left( \overline{X}\right) =Int^{KM}\left( \overline{X}\right)
=Int^{K}\left( \overline{X}\right) =\cup _{0<q\leq p}Int_{\left( p,q\right)
}^{KM}\left( \overline{X}\right) .
\end{equation*}
\end{corollary}

The following result can be applied to both Banach and quasi-Banach couples.

\begin{corollary}
If a $L$-convex quasi-Banach lattice couple $\overline{X}$ has the Calder\'{o}n-Mityagin property, then 
\begin{equation*}
Int\left( \overline{X}\right) =Int^{KM}\left( \overline{X}\right)
=Int^{K}\left( \overline{X}\right) =\bigcup _{0<q\leq p\leq 1}Int_{\left(
p,q\right) }^{KM}\left( \overline{X}\right) .
\end{equation*}
\end{corollary}

\begin{problem}
Let $\overline{X}$ be an arbitrary quasi-Banach couple. Then, the
definitions and Theorem \ref{th6.3} imply the following inclusions: 
\begin{equation*}
\bigcup _{0<q\leq p}Int_{\left( p,q\right) }^{KM}\left( \overline{X}\right)
\subseteq Int^{K}\left( \overline{X}\right) \subseteq Int^{KM}\left( 
\overline{X}\right) .
\end{equation*}%
The question is: To find necessary and sufficient conditions under which the
opposite inclusions hold. We note that, by Theorem \ref{th6.3}(b), the first
one can be reversed being restricted to mutually closed spaces.
\end{problem}

Applying Theorem \ref{th6.2} and Proposition \ref{pr6.8}, we get

\begin{theorem}
\label{Th-K-monotone-class} Let $0<q\leq p\leq 1$. Suppose that $\overline{X}%
=(X_{0},X_{1})$ is a $p$-convex quasi-Banach lattice couple and $X$ is a $%
(p,q)$-convex quasi-Banach lattice, which is $K$-monotone with respect to $%
\overline{X}$. Then, there exists a $(p,q)$-convex quasi-Banach lattice $E$
of measurable functions on $(0,\infty )$ such that $E\in Int\left(\overline{%
L^{\infty }}\right)$ and $X=\overline{X}_{E:K}.$
\end{theorem}

\begin{corollary}
\label{cor0} Let $\overline{X}=(X_0,X_1)$ be a $L$-convex quasi-Banach
lattice couple. Then, every $L$-convex quasi-Banach lattice $X$ that is $K$%
-monotone with respect to $\overline{X}$ can be representable as $\overline{X%
}_{E:K}$, where $E$ is some $L$-convex quasi-Banach function lattice on $%
(0,\infty)$ such that $E\in Int\left(\overline{L^{\infty }}\right)$.
\end{corollary}

\begin{proof}
We select $p\in (0,1]$ so that all spaces $X_{0},X_{1},X$ are $p$-convex and
then apply Theorem \ref{Th-K-monotone-class} together with Lemma \ref%
{properties of}.
\end{proof}

\begin{remark}
Similar results can be obtained also in a more general case of relative $K$%
-monotone spaces. Let $X$ and $Y$ be intermediate quasi-Banach lattices with
respect to quasi-Banach lattice couples $\overline{X}$ and $\overline{Y}$,
respectively. If $X$ and $Y$ are relative $K$-monotone spaces, then 
\begin{equation*}
X\subseteq \overline{X}_{\widehat{X}_{p,q}:K}\;\;\mbox{and}\;\;\overline{Y}_{%
\widehat{X}_{p,q}:K}\subseteq Y,
\end{equation*}%
provided that $\overline{X}$, $\overline{Y}$ are $p$-convex and $X$, $Y$ are 
$q$-convex, $0<q\leq p\leq 1$ (see details in \cite[Corollary 4.1.15]{BK91}
and \cite[Corollary 4.2]{Nil83}). Thus, a pair of relative $K$-monotone
spaces with respect to quasi-Banach lattice couples factors over a pair of $%
K $-spaces with the same parameter space. This result should be compared
with the corresponding property of relative interpolation spaces with
respect to Banach couples (see \cite[Theorem 13.XIV]{AG65} and \cite[Theorem
3.2.30]{BK91}).
\end{remark}

\begin{remark}
If $\overline{X}$ is a Banach lattice couple, Theorem \ref%
{Th-K-monotone-class} can be somewhat strengthened. For any $q$-normed space 
$X$ that is $K$-monotone with respect to $\overline{X}$ we have $X=\overline{%
X}_{E:K}$, where $E=Orb_{\overline{X}}^{q}\left( X;\overline{L^{\infty }}%
\right) $. Here, $Orb_{\overline{X}}^{q}\left( X;\overline{L^{\infty }}%
\right) $ denotes the space of all elements $f\in \Sigma \left( \overline{%
L^{\infty }}\right) $, which can be written as $f=\sum_{i}T_{i}x_{i}$, where 
$T_{i}:\overline{X}\rightarrow \overline{L^{\infty }},\left\Vert
T_{i}\right\Vert \leq 1$\thinspace\ and $\left\{ x_{i}\right\} \in
l^{q}\left( X\right) $ (see \cite{Nil83}).

Indeed, if $f\in \widehat{X}_{1,q},$ we can find a sequence $\left\{
x_{i}\right\} \in l^{q}\left( X\right) $ such that $\left\vert f(\cdot
)\right\vert \leq \sum_{i=1}^{\infty }K\left( \cdot ,x_{i};\overline{X}%
\right) $. Taking $\varepsilon >0$ and $T_{i}:\overline{X}\rightarrow \overline{%
L^{\infty }}$ with $T_{i}\left( x_{i}\right) (\cdot )=K\left( \cdot ,x_{i};%
\overline{X}\right) $ and $\left\Vert T_{i}\right\Vert <1+\varepsilon $, $%
i=1,2,\dots $ (see, for instance, \cite{BK91} and \cite{Nil83}), we select
then $h\in \Sigma \left( \overline{L^{\infty }}\right) $ such that $%
\left\vert h\right\vert \leq 1$ and 
\begin{equation*}
f(t)=h(t)\cdot \sum_{i=1}^{\infty }K\left( t,x_{i};\overline{X}\right)
,\;\;t>0.
\end{equation*}%
Hence, $f=\sum_{i=1}^{\infty }h\cdot T_{i}\left( x_{i}\right) $ and thus $%
f\in Orb_{\overline{X}}^{q}\left( X;\overline{L^{\infty }}\right) .$ This
implies that $\widehat{X}_{1,q}=Orb_{\overline{X}}^{q}\left( X;\overline{%
L^{\infty }}\right) $ isometrically.
\end{remark}

\begin{theorem}
\label{Th-K-monotone-class1} Let $\overline{X}=(X_{0},X_{1})$ be a $L$%
-convex quasi-Banach lattice couple. Then, for every $L$-convex quasi-Banach
lattice $X$ that is $K$-monotone with respect to $\overline{X}$ there is $%
p>0 $ such that 
\begin{equation*}
X=\left( \sum_{y\in X,\left\Vert y\right\Vert =1}Orb^{K}\left( y;\overline{X}%
\right) ^{p}\right) ^{1/p}.
\end{equation*}%
Moreover, 
\begin{equation*}
\left\{ K\left( \cdot ,x;\overline{X}\right) :\,x\in X\right\} \subset
\left( \sum_{y\in X,\left\Vert y\right\Vert =1}L^{\infty }\left( 1/\varphi
_{y}\right) ^{p}\right) ^{1/p},
\end{equation*}%
where $\varphi _{y}(\cdot )=K\left( \cdot ,y;\overline{X}\right) $.
\end{theorem}

\begin{proof}
Let $p>0$ is chosen so that all the spaces $X_{0}$, $X_{1}$ and $X$ are $p$%
-convex. Then, by Theorem \ref{Th-K-monotone-class}, $X=\overline{X}_{%
\widehat{X}_{p,p}:K}$. Hence, if $x\in X$, we have 
\begin{equation*}
K\left( t,x;\overline{X}\right) \leq \left( \sum_{i=1}^{\infty }K\left(
t,x_{i};\overline{X}\right) ^{p}\right) ^{1(p},\;\;t>0,
\end{equation*}%
where $x_{i}\in X$, $i=1,2,\dots $, and $\left( \sum_{i=1}^{\infty
}\left\Vert x_{i}\right\Vert ^{p}\right) ^{1/p}<\infty $. Clearly, we can
rewrite this as follows: 
\begin{equation*}
K\left( t,x;\overline{X}\right) =\left( \sum_{i=1}^{\infty }\left( \lambda
_{i}f\left( t\right) K\left( t,y_{i};\overline{X}\right) \right) ^{p}\right)
^{1/p},\;\;t>0,
\end{equation*}%
for some $y_{i}\in X$, $\lambda _{i}\in \mathbb{R}$, with $\Vert y_{i}\Vert
_{X}=1$, $\sum_{i=1}^{\infty }\left\vert \lambda _{i}\right\vert ^{p}<\infty 
$, and $f\in \Sigma \left( \overline{L^{\infty }}\right) $ such that $\Vert
f\Vert _{\Sigma \left( \overline{L^{\infty }}\right) }\leq 1$. Since $%
\lambda _{i}f\left( \cdot \right) K\left( \cdot ,y_{i};\overline{X}\right)
\in L^{\infty }\left( 1/K\left( \cdot ,y_{i};\overline{X}\right) \right) $, $%
i=1,2,\dots $, we have 
\begin{equation*}
K\left( \cdot ,x;\overline{X}\right) \in \left( \sum_{y\in X,\left\Vert
y\right\Vert =1}L^{\infty }\left( 1/\varphi _{y}\right) ^{p}\right) ^{1/p},
\end{equation*}%
and the second assertion of the theorem is proved.

From Theorem \ref{Th-K-monotone-class} it follows also that $\left\vert
x\right\vert =\left( \sum_{i=1}^{\infty }\left\vert x_{i}\right\vert
^{p}\right) ^{1/p}.$ Since 
\begin{equation*}
x_{i}\in \overline{X}_{\varphi _{x_{i}/\Vert x_{i}\Vert _{X}},\infty
:K}=Orb^{K}\left( x_{i}/\Vert x_{i}\Vert _{X};\overline{X}\right)
,\;\;i=1,2,\dots ,
\end{equation*}%
we get the first assertion.
\end{proof}

\section{Applications}

\label{Applications}

Interpolation properties of the couples of $L^{p}$-spaces, with $1\leq p\leq \infty $, are well known for a long time
(see e.g. \cite{BL76}). Recently, a similar investigation has been undertaken in a series of papers (see \cite{AsCwNi21}, \cite{CSZ}, \cite{Cad}, \cite{Ast-20}, \cite{CN17}) in the quasi-Banach setting. Results obtained in the previous sections allow us to push this research somewhat forward. Let us begin with the sequence case.

\begin{theorem}
\label{Th2} Let $0<p<q\leq \infty $. Then, every $K$-monotone quasi-Banach
sequence lattice $X$ with respect to the couple $\left(l^{p},l^{q}\right) $
is a $K$-space, i.e., $X=\left( l^{p},l^{q}\right) _{E:K}$ for some
quasi-Banach lattice $E\in Int\left( \overline{L^{\infty }}\right) $. If
additionally $q\geq 1$, the same holds for every quasi-Banach sequence
lattice $X$ such that $X\in Int\left(l^{p},l^{q}\right) $.
\end{theorem}

\begin{proof}
The first claim follows immediately from Theorem \ref{th6.2} and the fact that $\left( l^{p},l^{q}\right) $, $0<p<q\leq \infty $, is a $L$-convex quasi-Banach lattice couple. To get the second, one should use also Corollary 4.6 from \cite{AsCwNi21}, which asserts that in the case when $q\geq 1$ every interpolation quasi-Banach sequence space between $l^{p}$ and 
$l^{q}$ is $K$-monotone with respect to the couple $\left(
l^{p},l^{q}\right) $.
\end{proof}

Recall that a quasi-Banach sequence lattice $X$ is called \textit{symmetric}
if $X\subset l^\infty$ and the conditions $y_{k}^{\ast }\le x_{k}^{\ast }$, $%
k=1,2,\dots $, $x=(x_{k})_{k=1}^\infty\in X$ imply that $y=(y_{k})_{k=1}^%
\infty\in X$ and $\|y\|_{X}\le \|x\|_{X}$. Here, $(u_{k}^{\ast
})_{k=1}^{\infty }$ denotes the nonincreasing permutation of the bounded
sequence $(|u_{k}|)_{k=1}^{\infty }$ defined by 
\begin{equation*}
u_{k}^{\ast }:=\inf_{\mathrm{card}\,A=k-1}\sup_{i\in\mathbb{N}\setminus
A}|u_i|,\;\;k\in\mathbb{N}
\end{equation*}
(more detail about symmetric sequence spaces see in \cite{LT79-II} or \cite{KPS82}).

\begin{corollary}
\label{lsmall-cor1} Suppose $X$ is a symmetric quasi-Banach sequence space.
Then, for each sufficiently small $p>0$ and $E\in Int\left( \overline{L^{\infty }}\right) $ we have $X=\left(l^{p},l^{\infty
}\right) _{E:K}$.
\end{corollary}

\begin{proof}
Observe that, by \cite[Proposition~3.4]{AsCwNi21}, $X\in Int\left(
l^{p},l^{\infty }\right) $ for every sufficiently small $p>0$. It remains to
apply Theorem \ref{Th2}.
\end{proof}

From Theorem \ref{Th-K-monotone-class1} it follows

\begin{corollary}
Let $0<p<q\leq \infty $ and let $X$ be a $L$-convex quasi-Banach sequence
space such that $X$ is $K$-monotone with respect to the couple $\left(
l^{p},l^{q}\right) $. Then, we have 
\begin{equation*}
X=\left( \sum_{x\in X,\Vert x\Vert =1}\left( l^{p},l^{q}\right) _{\varphi
_{x},\infty :K}^{p}\right) ^{1/p},
\end{equation*}%
where $\left( l^{p},l^{q}\right) _{\varphi _{x},\infty :K}$ is the
generalized Marcinkiewicz space with $\varphi _{x}(t)=K\left(
t,x:l^{p},l^{q}\right) $, i.e., 
\begin{equation*}
\Vert y\Vert _{\left( l^{p},l^{q}\right) _{\varphi _{x},\infty
:K}}:=\sup_{t>0}\frac{K\left( t,y:l^{p},l^{q}\right) }{\varphi _{x}(t)}.
\end{equation*}
\end{corollary}

\begin{theorem}
\label{Th3} Let $0<p<q<1$ and let $X$ be a quasi-Banach sequence lattice.
Then $X\in Int\left( l^{p},l^{q}\right) $ if and only if $X=\left( l^{p},l^{1}\right)_{E:K}$, for some quasi-Banach lattice $E\in Int\left( \overline{L^{\infty }}\right) $ such
that $E\subseteq L^{q}\left( s^{-\theta },\frac{dt}{t}\right)$, where $1/q=\left( 1-\theta \right) /p+\theta $.
\end{theorem}

\begin{proof}
Suppose first that $X\in Int\left( l^{p},l^{q}\right) $. Since $l^{q}=\left( l^{p},l^{1}\right) _{\theta,q:K}=\left( l^{p},l^{1}\right) _{E_{1}:K}$, where $1/q=\left( 1-\theta \right) /p+\theta $ and $E_{1}=L^{q}\left(s^{-\theta },\frac{dt}{t}\right)$ (see e.g. \cite[Theorem~5.2.1]{BL76}), we have $X\in Int\left(l^{p},l^{1}\right) .$ Therefore, from Theorem \ref{Th2} it follows that $X=\left( l^{p},l^{1}\right) _{E_{0}:K}$ for some $E_{0}\in Int\left(\overline{L^{\infty }}\right) $. Moreover, $X\subseteq l^{q}$, and hence 
$$
X=\left( l^{p},l^{1}\right)_{E_{0}:K}\cap l^{q}=\left( l^{p},l^{1}\right)_{E_{0}:K}\cap \left( l^{p},l^{1}\right) _{E_{1}:K}=\left( l^{p},l^{1}\right)_{E_{0}\cap E_{1}:K}.$$
Thus, the desired result holds for $E:=E_{0}\cap E_{1}\in Int\left( \overline{L^{\infty }}\right) $.

Conversely, if $X=\left( l^{p},l^{1}\right)_{E:K}$, with $E\subseteq E_{1}$, then $X\subseteq \left( l^{p},l^{1}\right) _{E_{1};K}=l^{q}$, and then, by \cite[Corollary~3.3]{AsCwNi21}, $X\in Int\left( l^{p},l^{q}\right) $.
\end{proof}

\begin{proposition}
\label{right K-prop} Let $0\leq p<q\leq \infty $. Then, the couples $\left(
l^{p},l^{q}\right) $ and $\overline{L^{\infty }}$ have the uniform relative
Calder\'{o}n-Mityagin property if and only if $p\geq 1.$
\end{proposition}

\begin{proof}
If $p\geq 1$ (i.e., in the Banach case), the desired result is well-known, see, for instance, \cite{CwPe81} or \cite[Theorem 4.4.16]{BK91}.

For the converse, we may assume that $p<1$. Suppose that the couples $\left(
l^{p},l^{q}\right) $ and $\overline{L^{\infty }}$ have relative 
Calder\'{o}n-Mityagin property. Observe that for each $x\in l^{q}$ we have $K\left(\cdot ,x;l^{p},l^{q}\right) \in \Sigma \left( \overline{L^{\infty }}\right) $
and 
\begin{equation*}
K\left( t,K\left( \cdot ,x;l^{p},l^{q}\right) ;\overline{L^{\infty }}\right)
=K\left( t,x;l^{p},l^{q}\right) ,\;\;t>0
\end{equation*}%
(see Section \ref{Prel1}). Therefore, there exists $\lambda >0$ such that
for every $x\in l^{q}$ we can find a bounded linear operator $T:\left(
l^{p},l^{q}\right) \rightarrow \overline{L^{\infty }}$ satisfying $T\left(
x\right) =K\left( \cdot ,x;l^{q},l^{p}\right) $ and $\left\Vert T\right\Vert
\leq \lambda .$ Since $0\leq p<1,$ then $T$ has an extension $\widetilde{T}$ to
the Banach envelope of $l^{p}$, i.e., to $l^{1}$, such that $\Vert 
\widetilde{T}\Vert _{l^{1}\rightarrow L^{\infty }}=\Vert {T}\Vert $ (see
e.g. \cite{KPR84} or \cite{AsCwNi21}). Hence, $\widetilde{T}:\left(
l^{1},l^{q}\right) \rightarrow \overline{L^{\infty }}.$ Thus, for all $t>0$
and $x\in l^{p}\subset l^{1}$ 
\begin{equation*}
K\left( t,x;l^{p},l^{q}\right) =\left\vert T\left( x\right) \left( t\right)
\right\vert =|\widetilde{T}\left( x\right) \left( t\right) |\leq \Vert 
\widetilde{T}\Vert \left\Vert x\right\Vert _{l^{1}}\leq \lambda \left\Vert
x\right\Vert _{l^{1}}.
\end{equation*}%
Letting $t\rightarrow \infty $ we infer%
\begin{equation*}
\left\Vert x\right\Vert _{l^{p}}=\sup_{t>0}K\left( t,x;l^{p},l^{q}\right)
\leq \lambda \left\Vert x\right\Vert _{l^{1}},
\end{equation*}%
which is a contradiction, because of the assumption $p<1$.
\end{proof}

\begin{remark}
In the previous result, the couple $\overline{L^{\infty }}$ may be replaced
with other Banach couples having the universal right $K$-property\footnote{%
A Banach couple $\overline{W}=(W_{0},W_{1})$ has the universal right $K $%
-property if for every Banach couple $\overline{X}=(X_{0},X_{1})$ the
couples $\overline{X}$ and $\overline{W}$ have the uniform relative Calder\'{o}n-Mityagin property.}, for instance, $\left( l^{\infty },l^{\infty
}(2^{-k})\right) $ or $\left( C,C(1/t)\right) $, where $C$ is the space of
all bounded continuous functions on $(0,\infty )$ with the natural $\sup $%
-norm (see, for instance, \cite{CwPe81}).
\end{remark}

Observe that the couple $\left( l^{p},l^{q}\right) $, $0< p<q\leq \infty $%
, is a uniform Calder\'{o}n-Mityagin couple if and only if $q\geq 1$
\cite[Corollary 5.4]{AsCwNi21}. In contrast to that, by Theorem 1.1 of the paper \cite{CSZ} (see also \cite[Remark~4.7]{AsCwNi21}), one can readily see
that the couple $\left( L^{p},L^{q}\right) $ of Lebesgue measurable
functions on $(0,\infty )$ is a uniform Calder\'{o}n-Mityagin couple for all parameters $0<p<q\leq \infty $. Therefore, in the same way as Theorem \ref{Th2}, we get the following result.

\begin{theorem}
\label{Th4} Let $0<p<q\leq \infty $. Then every quasi-Banach lattice $X$ such that $X\in Int\left( L^{p},L^{q}\right) $, where the underlying measure space is $(0,\infty)$ with the Lebesgue measure, is a $K$-space, i.e., $X=\left(L^{p},L^{q}\right) _{E:K}$ for some quasi-Banach lattice $E\in Int\left( \overline{L^{\infty }}\right) $.
\end{theorem}

\end{document}